\documentclass[10pt,oneside,a4paper,reqno]{amsart}  %%(fold)
\usepackage{amssymb,graphicx,mathrsfs,xcolor}
\usepackage[colorlinks=true, linkcolor=magenta, citecolor=cyan, urlcolor=blue]{hyperref}
\numberwithin{equation}{section}\newtheorem{theorem}{Theorem}[section]
\newtheorem{corollary}[theorem]{Corollary}\newtheorem{lemma}[theorem]{Lemma}

\newtheorem{proposition}[theorem]{Proposition}\theoremstyle{remark}
\newtheorem{remark}{Remark}[section]
\theoremstyle{definition}

\DeclareMathOperator{\supp}{supp}
\newcommand{\bra}[1]{\langle #1 \rangle}
\newcommand{\one}[1]{\mathbf{1}_{#1}}

\def\Xint#1{\mathchoice
{\XXint\displaystyle\textstyle{#1}}%
{\XXint\textstyle\scriptstyle{#1}}%
{\XXint\scriptstyle\scriptscriptstyle{#1}}%
{\XXint\scriptscriptstyle\scriptscriptstyle{#1}}%
\!\int}
\def\XXint#1#2#3{{\setbox0=\hbox{$#1{#2#3}{\int}$ }
\vcenter{\hbox{$#2#3$ }}\kern-.58\wd0}}
\newcommand\dashint{\Xint-}

\title%
[Weighted estimates]%
{Weighted $L^{p}$ estimates for powers of selfadjoint operators}
\date{\today}    %%% ''\date{}'' to omit date

\author{Federico Cacciafesta}
\address{Federico Cacciafesta: 
SAPIENZA --- Universit\`a di Roma,
Dipartimento di Matematica, 
Piazzale A.~Moro 2, I-00185 Roma, Italy}
\email{cacciafe@mat.uniroma1.it}
\author{Piero D'Ancona}
\address{Piero D'Ancona: 
SAPIENZA --- Universit\`a di Roma,
Dipartimento di Matematica, 
Piazzale A.~Moro 2, I-00185 Roma, Italy}
\email{dancona@mat.uniroma1.it}


\begin{document}\maketitle

\begin{abstract}
  We prove  $L^{p}$ and weighted $L^{p}$ estimates for bounded 
  functions of a selfadjoint operator satisfying
  a pointwise gaussian estimate for its heat kernel. As an application,
  we obtain weighted estimates for fractiional powers 
  of an electromagnetic Schr\"odinger operator with singular 
  coefficienta. The proofs are based on a modification of
  techniques due to Hebisch \cite{Hebisch90-b} and
  Auscher and Martell \cite{AuscherMartell06-a}.
\end{abstract}

%%%%%   INDEX (toc)
%\tableofcontents
%(end)

\section{Introduction}\label{sec:intro} %(fold)

The question of $L^{p}$ estimates for functions of a selfadjoint
operator is a delicate one. Indeed, even for a Schr\"odinger
operator $H=-\Delta+V(x)$ with a nonnegative
potential $V\in C^{\infty}_{c}$,
and a bounded smooth function $f(t)$, the operator $f(H)$ defined
via spectral theory does not have in general a smooth kernel and
hence does not fall within the scope of the Calder\`on-Zygmund
theory. The first to overcome this difficulty
was Hebisch \cite{Hebisch90-b} who proved the following result;
we use the notation
\begin{equation*}
  S_{\lambda}f(t)=f(\lambda t),\qquad \lambda>0
\end{equation*}
for the scaling operator, and we denote by $H^{s}$ the usual 
$L^{2}$--Sobolev space.

\begin{theorem}[\cite{Hebisch90-b}]\label{the:hebischth}
  Let $H$ be a nonnegative selfadjoint operator on
  $L^{2}(\mathbb{R}^{n})$ satisfying a
  gaussian estimate
  \begin{equation}\label{eq:ggauss}
    0\le e^{-tH}(x,y)\le C t^{-\frac n2}e^{-\frac{|x-y|^{2}}{4t}},
  \end{equation}
  let $\phi\in C^{\infty}_{c}(\mathbb{R}^{+})$ be
  a nonzero cutoff, and assume the function
  $F(s)$ on $\mathbb{R}^{+}$ satisfies
  \begin{equation}\label{eq:horm}
    \sup_{t>0}\|\phi S_{t}F\|_{H^{a}}<\infty
    \quad\text{for some }\quad a>\frac{n+1}{2}.
  \end{equation}
  Then the operator $F(H)$ is bounded from $L^{1}$ to
  $L^{1,\infty}$ and on any $L^{p}$, $1<p<\infty$.
\end{theorem}

Theorem \ref{the:hebischth} raises 
a few interesing questions concerning
the optimality of the assumptions and the possibility of
\emph{weighted} $L^{p}$ estimates for suitable classes
of operators. In the case $H=-\Delta$,
the classical H\"ormander mutliplier theorem
requires only $a>n/2$ in \eqref{eq:horm}, and in this sense
the result is not optimal. Indeed, sharper results were
obtained for bounded functions of homogeneous Laplace operators
acting on homogeneous groups or on groups of polynomial growth
(see
\cite{De-MicheleMauceri87-a},
\cite{Christ91-a},
\cite{MauceriMeda90-a},
\cite{Alexopoulos94-a}). 
In these results the conditions on the
function $F$
were sharpened to
\begin{equation}\label{eq:condsharp1}
  \sup_{t>0}\|\phi S_{t}F\|_{H^{a}_{p}}<\infty
  \quad\text{for some }\quad a>\frac{n}{2}
\end{equation}
where $H^{a}_{p}$ is the Sobolev space with norm
$\|(1-d^{2}/dx^{2})^{\frac a2}f\|_{L^{p}}$, and $p$ is equal to
2 or $\infty$.
The criticality of the order $a=n/2$ was proved
by Sikora and Wright \cite{SikoraWright01-a}
in the special case of imaginary powers 
$L^{iy}$, with $L$ a positive selfadjoint operator of the form
\begin{equation*}
  L=-\sum \partial_{i}a_{ij}\partial_{j}.
\end{equation*}
They obtained
\begin{equation}\label{eq:sikw}
  \|L^{iy}\|_{L^{1}\to L^{1,\infty}}\simeq(1+|y|)^{\frac n2}
\end{equation}
provided $L$ satisfies, besides the gaussian estimate,
a \emph{finite speed of propagation} property, meaning that
the operator $\cos(t \sqrt{L})$ has an integral kernel
$K_{t}(x,y)$
supported in the ball $|x-y|\le t$ for all $t\ge0$.
Notice that
the norm \eqref{eq:horm} for $a=n/2$ and $F(s)=s^{iy}$ grows
precisely like $(1+|y|)^{\frac n2}$.
It was later remarked by Sikora
\cite{Sikora04-a} that the finite speed of propagation
is redundant and actually equivalent to a weaker
Gaussian bound, the so-called Davies-Gaffey
$L^{2}$ estimate (see Remark \ref{rem:finitesp} below).

Condition \eqref{eq:condsharp1} was further improved 
by Duong, Ouhabaz and Sikora
\cite{DuongOuhabazSikora02-a}. They obtained a general result
for functions of a selfajoint, positive operator $L$
on $L^{2}(X,\mu)$ where $X$ is any open subset of a
space of homogeneous type, $\mu$ a doubling measure,
and $L$ satisfies a generalized pointwise gaussian estimate
analogous to \eqref{eq:ggauss}. In particular they obtained that
if $F$ is bounded and satisfies
\eqref{eq:condsharp1} with $p=\infty$, then
$F(L)$ is of weak type $(1,1)$ and bounded on
all $L^{q}$, $1<q<\infty$. On the other hand, if
\eqref{eq:condsharp1} holds for some $p\in[2,\infty)$,
the same result holds provided $L$ satisfies an additional
a priori condition of Plancherel type on the kernel
of $F(\sqrt{L})$; see \cite{DuongOuhabazSikora02-a}
for further results and an extensive bibliography.

Our main purpose here is to extend these results, at
least in the euclidean setting, to the case of 
\emph{weighted} $L^{p}$ spaces. However, in order to
develop our techniques, we shall first prove a precised
version of Theorem \ref{the:hebischth}, building on
the ideas of
\cite{Hebisch90-b}, \cite{SikoraWright01-a}.
Concerning the operator $H$, as in
Hebisch' result, we shall only require a gaussian bound;
for further reference we state the condition as

\textsc{Assumption (H)}. $H$ is a nonnegative selfadjoint
operator on $L^{2}(\mathbb{R}^{n})$ satisfying
a gaussiam heat kernel estimate
\begin{equation}\label{eq:heatkernH}
  |p_{t}(x,y)|\le \frac{K_{0}}{t^{n/2}}
    e^{-|x-y|^{2}/(dt)},\qquad d>0.
\end{equation}
A rescaling $H\to\lambda H$ shows that it is
not restrictive to assume $d=1$.

\begin{remark}\label{rem:condHsuff}
  In section \ref{sec:electromagnetic} we shall 
  exhibit a wide class of operators satisfying (H),
  namely the electromagnetic Schr\"odinger operators
  \begin{equation}\label{eq:H}
    H=(i \nabla-A(x))^{2}+V(x)
  \end{equation}
  under very weak conditions on the potentials: more precisely,
  it is sufficient to assume that $A\in L^{2}_{loc}$ and
  that $V$ is in the Kato class
  with a negative part $V_{-}$ small enough.
  For related results on magnetic Schr\"odinger operators
  see also \cite{Ben-Ali09-a}.
\end{remark}

In order to express the smoothness conditions in an
optimal way, we shall introduce two norms on functions
defined on the positive real line. In the rest of the paper
we fix a cutoff $\psi\in C^{\infty}_{c} (\mathbb{R})$
with support in $[-2,2]$ and equal to $1$ on $[-1.1]$,
and denote with $\phi$ the function, supported in $[1/2,2]$,
\begin{equation}\label{eq:phi}
  \phi(s)=
  \begin{cases}
     \psi(s)-\psi(2s)&\text{if $ s>0 $,}\\
     0&\text{if $ s\le0 $.}
  \end{cases}
\end{equation}
As a consequence, notice the identities for $s>0$
\begin{equation}\label{eq:expa}
  \psi(s)=\sum_{k>0}\phi(2^{k}s),\qquad
  1-\psi(s)=\sum_{k\le0}\phi(2^{k}s).
\end{equation}
Then, writing $\bra{\xi}=(1+|\xi|^{2})^{1/2}$,
the norms $\mu_{a},\mu'_{a}$ will be defined as
\begin{equation}\label{eq:assbddg}
  \mu_{a}(g)=
  \sup_{\lambda>0}\|\bra{\xi}^{a}\mathcal{F}[\phi(s)
  S_{\lambda}g]\|_{L^{1}},\qquad
  \mu_{a}'(g)=
  \sup_{\lambda>0}\|\bra{\xi}^{a}\mathcal{F}[s\phi(s)
  S_{\lambda}g]\|_{L^{1}}.
\end{equation}

\begin{remark}\label{rem:sob}
  It is easy to control $\mu_{a}$ with ordinary
  Besov or Sobolev norms:
  \begin{equation}\label{eq:besovbd}
    \mu_{a}(g)\le c(n)
    \sup_{t>0}
    \| \phi S_{t}g\|_{B^{a+\frac{1}{2}}_{2,1}}\le
    c(n,\epsilon)
    \sup_{t>0}
    \| \phi S_{t}g\|_{H^{a+\frac{1}{2}+\epsilon}},
    \qquad \epsilon>0.
  \end{equation}
  The last norm in \eqref{eq:besovbd} is the one used
  in Theorem \ref{the:hebischth}, and using $\mu_{a}$ instead
  allows to eliminate the $1/2+$ loss of smoothness 
  in Hebisch' result.
\end{remark}

Our first result is the following:

\begin{theorem}\label{the:hebweakest}
  Let $H$ be an operator satisfying (H) and
  $g(s)$ a function on $\mathbb{R}^{+}$ 
  with $\mu=\mu_{\sigma}(g)<\infty$
  for some $\sigma>n/2$. Then the following weak $(1,1)$ estimate
  holds:
  \begin{equation}\label{eq:hebweak}
    \|g(\sqrt{H})f\|_{L^{1,\infty}}\le
    C\|f\|_{L^{1}},\qquad
    C=c(n,\sigma)K_{0}^{4}(1+\mu+\|g\|_{L^{\infty}}^{2}),
  \end{equation}
  and for all $1<p<\infty$, with the same $C$,
  \begin{equation}\label{eq:hebLp}
    \|g(\sqrt{H})f\|_{L^{p}}\le
    6C \left(p+(p-1)^{-1}\right)
    \|f\|_{L^{p}}
  \end{equation}
  If in addition we assume that for some $q>1$ the following
  estimate holds:
  \begin{equation}\label{eq:hebassnew}
    \|\sqrt{H}g(\sqrt{H})f\|_{L^{q}}\le C_{q}\|\nabla f\|_{L^{q}},
  \end{equation}
  and $\mu'=\mu_{\sigma}'(g)<\infty$ for a $\sigma>1+n/2$,
  then we have also
  \begin{equation}\label{eq:hebweaknew}
    \|\sqrt{H}g(\sqrt{H})f\|_{L^{1,\infty}}\le
    C\|\nabla f\|_{L^{1}},\qquad
    C=c(n,\sigma,C_{q})K_{0}^{4}(1+\mu'+\|g\|_{L^{\infty}}^{2}),
  \end{equation}
  and for all $1<p\le q$, with the same $C$,
  \begin{equation}\label{eq:hebLpnew}
    \|\sqrt{H}g(\sqrt{H}) f\|_{L^{p}}\le
    \frac{c(q)}{p-1}C
    \|\nabla f\|_{L^{p}}.
  \end{equation}
\end{theorem}

\begin{remark}\label{rem:comparison}
  As mentioned above, in \cite{DuongOuhabazSikora02-a}
  it was proved that the weak $(1,1)$ estimate holds
  under the sole assumption
  \begin{equation*}
    \sup_{t>0}\|\phi S_{t}g\|_{H^{a}_{\infty}}<\infty
  \end{equation*}
  for some $a>n/2$
  (see Theorem 3.1 and Remark 1 in that paper). Since 
  obviously
  \begin{equation*}
    \sup_{t>0}\|\phi S_{t}g\|_{H^{a}_{\infty}}\lesssim
    \mu_{a}(g),
  \end{equation*}
  we see that estimate
  \eqref{eq:hebweak} can be obtained as a special case of that
  result, with a slightly different form of the constant
  which we made explicit in terms of the gaussian constant $K_{0}$.
  On the other hand, 
  estimate \eqref{eq:hebLpnew}, which uses
  Auscher's Calderon-Zygmund decomposition for Sobolev
  functions \cite{Auscher07-a}, seems to be new.
\end{remark}

\begin{remark}\label{rem:compareheb}
  As evidenced by the previous discussion,
  the constant in \eqref{eq:hebweak} is close to optimal
  in the following sense: if we choose $g(s)=s^{2iy}$, we have
  \begin{equation*}
    \mu_{a}(g)\le C(1+|y|)^{a},\qquad a\ge0;
  \end{equation*}
  (the proof is trivial for integer values of
  $a$ and follows by interpolation for
  real values). This implies that, for all $\epsilon>0$
  and $1<p<\infty$,
  \begin{equation}\label{eq:Hiy}
    \|H^{iy}f\|_{L^{p}}\le C(p,n,\epsilon)(1+|y|)^{\frac n2+\epsilon}
    \|f\|_{L^{p}}
  \end{equation}
  which is close to the optimal bound \eqref{eq:sikw}.
  Notice also that the strict condition $\sigma>n/2$ can be further
  optimized to a logarithmic condition, but we prefer not
  to pursue this idea here.
\end{remark}

After it was made clear by the results of Hebisch
and others that kernel smoothness
is not a necessary condition for $L^{p}$ boundedness,
alternative weaker conditions where thoroughly investigated,
also in connection with the Kato problem. A fairly
complete answer was given by Auscher and Martell 
who developed a general theory in a series of papers 
(see in particular \cite{AuscherBen-Ali07-a}, 
\cite{AuscherMartell06-a} and the references therein).
By combining the techniques of Auscher and Martell
with ideas from the proof of Theorem \ref{the:hebweakest}, we are
able to extend the previous estimates
to weighted spaces $L^{p}(w)$. 
In the following we use the notation
\begin{equation*}
  \|f\|_{L^{p}(w)}=\left(\int|f|^{p}w(x)dx\right)^{1/p}
\end{equation*}
and we recall that a measurable function $w(x)>0$ belongs to the
\emph{Muckenhoupt class} $A_{p}$, $1<p<\infty$,
if the quantity
\begin{equation}\label{eq:Ap}
  \|w\|_
  {A_{p}}=\sup_{Q \ \text{cube}}
  \left(\dashint_{Q}w \right)
  \left(\dashint_{Q}w^{1-p'}\right)^{p/p'}<\infty.
\end{equation}
is finite. Then the main result of this paper is

\begin{theorem}\label{the:1}
  Let $H$ be an operator satisfying (H), and let
  $g$ be a bounded function on $\mathbb{R}^{+}$ such that
  $\mu=\mu_{\sigma}(g)$ is finite for some $\sigma>n/2$.
  Then, given any $1<p<\infty$ and any weight $w\in A_{p}$,
  the operator $g(\sqrt{H})$ satisfies, for all $1<q<\infty$
  with $q>p\cdot\max\{1,n/\sigma\}$
  \begin{equation}\label{eq:FH}
    \|g(\sqrt{H})f\|_{L^{q}(w)}\le 
    c(n,\sigma,p,\psi,w)K_{0}^{1+2p^{2}} 
    (1+\mu+\|g\|_{L^{\infty}}^{2})\cdot q \cdot
    \|f\|_{L^{q}(w)}.
  \end{equation}
\end{theorem}

\begin{remark}\label{rem:Ap}
  It is well known that if $w\in A_{p}$ for some $p>1$, then we have 
  also $w\in A_{p-\epsilon}$ for some $\epsilon>0$ depending only
  on $\|w\|_{A_{p}}$ (for a quantitative estimate of $\epsilon$
  see \cite{Kinnunen98-a}). Thus in the statement of Theorem \ref{the:1}
  the condition on $q$ can be relaxed to
  \begin{equation}\label{eq:condq}
    q>(p-\epsilon)\max\left\{1,\frac n \sigma\right\}.
  \end{equation}
  In particular, if $\sigma\ge n$, we have that $g(\sqrt{H})$
  is bounded on $L^{q}(w)$ for all $w\in A_{p}$ and all $q>p-\epsilon$,
  which includes the case $q\ge p$.
\end{remark}

\begin{remark}\label{rem:origmot}
  The original motivation for the present work was the need for an
  estimate
  \begin{equation}\label{eq:weiginitial}
    \|\bra{x}^{-1-\epsilon}H^{\theta}g\|_{L^{2}}\le
    C(V)  \|\bra{x}^{-1-\epsilon}(-\Delta)^{\theta}g\|_{L^{2}},\qquad
    \theta=\frac14,\qquad
    H=-\Delta+V(x)
  \end{equation}
  for fractional powers of a selfadjoint Schr\"odinger operator
  $H$, with explicit bounds on the constant $C(V)$. 
  For the case $\theta=1/2$, and operators in divergence
  form, similar estimates are included in the results of
  \cite{AuscherMartell06-a} (see also \cite{AuscherBen-Ali07-a})
  concerning reverse estimates for square roots of an elliptic operator.
  However, other values of $\theta$, different forms of $H$,
  and the need for precise bounds on the constant, forced us to
  go beyond the existing theory.

  It may be interesting to recall briefly
  the line of investigation leading to \eqref{eq:weiginitial}.
  An analysis of the dispersive properties of Schr\"odinger
  equations on non-flat waveguides (i.e. perturbations of
  domains of the form $\mathbb{R}^{n}\times \Omega$
  with $\Omega$ a bounded open set, see 
  \cite{DanconaRacke10-a} for details)
  leads to a family of
  perturbed Schr\"od\-in\-ger equations
  \begin{equation}\label{eq:probj}
    iu_{t}+\Delta_{x}u-V_{j}(x)u=0,\qquad
    u(0,x)=f_{j}(x),\qquad j\ge1,\ x\in \mathbb{R}^{n}.
  \end{equation}
  Here $u=u_{j}$ is a component of the expansion in
  a distorted Fourier series of a
  function $u(t,x,y)=\sum \phi_{j}(y)u_{j}(t,x)$.
  Writing for short
  $H_{j}=-\Delta+V_{j}$ and representing the solution as
  \begin{equation*}
    u_{j}=e^{it H_{j}}f_{j},
  \end{equation*}
  one expects to estimate each component separately
  and sum over $j$. Notice that a precise bound on the
  growth in $j$ of the constants is essential, since
  this will translate into $y$-derivatives after summing over $j$.
  To this end we can use
  \emph{smoothing estimates} of the form
  \begin{equation}\label{eq:smoo12}
    \|\bra{x}^{-1-\epsilon}(-\Delta)^{1/4}e^{itH_{j}}f_{j}
          \|_{L^{2}_{t}L^{2}_{x}}\le
       C\|H_{j}^{1/4}f_{j}\|_{L^{2}}.
  \end{equation}
  which can be proved by multiplier techniques and
  give a complete control on the growth of the constants,
  and then deduce, in a standard way,
  \emph{Strichartz estimates}, which are the basci tool for
  applications to nonlinear problems. This is possible
  provided we can ``simplify'' the
  powers of $-\Delta$ and $H_{j}$ appearing in \eqref{eq:smoo12}
  and obtain the $L^{2}$--level estimate
  \begin{equation}\label{eq:smoo}
    \|\bra{x}^{-1-\epsilon}e^{itH_{j}}f_{j}\|_{L^{2}_{t}L^{2}_{x}}\le
       C\|f_{j}\|_{L^{2}}.
  \end{equation}
  But of course $(-\Delta)^{1/4}$ and $e^{itH_{j}}$
  do not commute, hence this step is not trivial.
  We need a \emph{weighted $L^{2}$ estimate} of the form
  \begin{equation}\label{eq:weigbase}
    \|\bra{x}^{-1-\epsilon}H_{j}^{1/4}g\|_{L^{2}}\le
    C(V_{j})  \|\bra{x}^{-1-\epsilon}(-\Delta)^{1/4}g\|_{L^{2}}
  \end{equation}
  so that we can replace $(-\Delta)^{1/4}$ by 
  $H_j^{1/4}$ in the LHS of \eqref{eq:smoo12},
  commute it with $e^{itH_{j}}$, and obtain
  \eqref{eq:smoo}. From the previous discussion, it is clear that
  we need also a precise control on the constant in \eqref{eq:weigbase}.
  
  Our weighted estimates, via complex interpolation, allow us
  to give a partial answer to
  the original problem \eqref{eq:weiginitial}. Indeed,
  for a Schr\"odinger operator on $\mathbb{R}^{n}$, $n\ge3$
  \begin{equation*}
    H=-\Delta+V(x),\qquad V\ge0
  \end{equation*}
  we obtain the bounds
  \begin{equation}\label{eq:explV}
    \|\bra{x}^{-s}H^{\theta}f\|_{L^{p}}\le
    C(n,p,s)\cdot
    \left[
    1+\|V\|_{L^{n/2,\infty}}
    \right]^{\theta}\cdot
    \|\bra{x}^{-s}(-\Delta)^{\theta}f\|_{L^{p}}
  \end{equation}
  for all $\theta,p,s$ in the range
  \begin{equation*}
    0\le\theta\le 1,\qquad
    1<p<\frac{n}{2 \theta},\qquad
    s>-\frac np.
  \end{equation*}
  More generally, we can prove
  (see the beginning of Section \ref{sec:electromagnetic} for 
  the definition of Kato classes):
  
  \begin{corollary}\label{the:2}
    Consider the operator
    \begin{equation*}
      H=(i\nabla-A(x))^{2}+V(x)
    \end{equation*}
    on $L^{2}(\mathbb{R}^{n})$, $n\ge3$, under the assumptions that
    $A\in L^{2}_{loc}(\mathbb{R}^n, \mathbb{R}^n)$,
    $V_{+}=\max\{V,0\}$ is of Kato class,
    $V_{-}=\max\{-V,0\}$ has a small Kato norm
    \begin{equation}\label{eq:smallkato2}
      \|V_{-}\|_{K}<c_{n}=
      \frac{\pi^{\frac n2}}{\Gamma\left(\frac n2-1\right)},
    \end{equation}
    and
    \begin{equation}\label{eq:ass2A}
      |A|^{2}-i \nabla \cdot A+V\in L^{n/2,\infty},\qquad
      A\in L^{n,\infty}.
    \end{equation}
    Then $H$ satisfies assumption (H), and
    for all $0\le \theta\le 1$ the following estimate holds:
    \begin{equation}\label{eq:fract}
       \|H^{\theta}f\|_{L^{p}(w)}\le
       C\|(-\Delta)^{\theta} f\|_{L^{p}(w)}
    \end{equation}
    for all weights $w\in A_{p}$ provided
    \begin{equation*}
      1<p<\frac{n}{2 \theta}.
    \end{equation*}
    The constant in \eqref{eq:fract} has the form
    \begin{equation*}
      C=\frac{ C(n,p,w)}{(1-\|V_{-}\|_{K}/c_{n})^{c(p)}}
       \Bigl[1+\||A|^{2}-i \nabla \cdot A+V\|_{L^{n/2,\infty}}+
       \|A\|_{L^{n,\infty}}\Bigr]^{\theta}.
    \end{equation*}
  \end{corollary}
\end{remark}

The paper is organized as follows. In section \ref{sec:kernel} we
build the necessary kernel estimates for functions of an
operator and apply them to the proof of the $L^{p}$
estimates of Theorem \ref{the:hebweakest};
sections \ref{sec:bddfunct}
is devoted to the proof of the main result, Theorem
\ref{the:1}, concerning weighted $L^{p}$ estimates;
the application
to magnetic Schr\'odinger operators is contained in sections
\ref{sec:electromagnetic} and
\ref{sec:th2}. We added an appendix containing
a slightly adapted version of the Auscher-Martell maximall lemma
in order to make the paper self contained.
In forthcoming papers we plan to apply our estimates to questions
of local smoothing and dispersion for evolution equations, in the
spirit of \cite{DanconaRacke10-a}, \cite{DanconaPierfeliceRicci10-a}.

\textit{Acknowledgments}. We would like to thank Dr.~The Anh Bui
for his useful remarks on the first version of the paper (see
\cite{Bui11-a} for related results). We are also grateful to
the Referee whose remarks led to substantial improvements in our
results.

% end of section Introduction (end)

\section{Kernel estimates and proof of Theorem \ref{the:hebweakest}}
\label{sec:kernel}  %(fold)

Throughout the proof, $\phi$ and $\psi$ are the functions
fixed in \eqref{eq:phi}--\eqref{eq:expa}. 
Given an operator $A$ with kernel $A(x,y)$, we denote its 
Schur norm with
\begin{equation*}
  \|A\|=\|A(x,y)\|\equiv
  \max\left\{
  \sup_{x}\int |A(x,y)|dy,\quad
  \sup_{y}\int |A(x,y)|dx
  \right\};
\end{equation*}
notice the product inequality
\begin{equation}\label{eq:prodin}
  \|AB\|\le\|A\|\cdot\|B\|
\end{equation}
which follows from the identity
\begin{equation}\label{eq:conv}
  (AB)(x,y)=\int A(x,z)B(z,y)dy.
\end{equation}
Following 
\cite{Hebisch90-b}, for any nonnegative function $w(x)$
on $\mathbb{R}^{n}$ 
we can define a weighted version of the above norm as
\begin{equation}\label{eq:weighH}
  \|A\|_{w}=\|A(x,y)w(x-y)\|.
\end{equation}

\begin{remark}\label{rem:finitesp}
  In the proof of the following Lemma we shall use the
  finite speed of propagation property of the
  kernel $\cos(\xi \sqrt{H})(x,y)$, namely the property
  \begin{equation}\label{eq:fsp}
    \cos(t \sqrt{H})(x,y)=0
    \quad\text{for $|x-y|>t\ge0$.}
  \end{equation}
  Adam Sikora in \cite{Sikora04-a} proved the remarkable
  fact that \eqref{eq:fsp} is equivalent to the following
  estimate:
  for all functions $f_{1},f_{2}$ supported in the balls
  $B(x_{1},r_{1})$ and $B(x_{2},r_{2})$ 
  respectively, and for any $r$ with
  \begin{equation}\label{eq:distr}
    |x_{1}-x_{2}|-(r_{1}+r_{2})>r\ge0
  \end{equation}
  one must have
  \begin{equation}\label{eq:davg}
    \left|(e^{-tH}f_{1},f_{2})_{L^{2}}\right|\le C
    e^{-r^{2}/t}\|f_{1}\|_{L^{2}}\|f_{2}\|_{L^{2}}.
  \end{equation}
  Estimates of the form \eqref{eq:davg} are usually
  called $L^{2}$ \emph{estimates} of
  \emph{Davies-Gaffey} type. 
  Notice that the pointwise estimate in assumption (H) 
  implies immediately \eqref{eq:davg} and hence \eqref{eq:fsp}.
  
  For the sake of completeness, we recall here the elementary
  argument from \cite{Sikora04-a} which allows to deduce
  \eqref{eq:fsp} from \eqref{eq:davg}. 
  Let $f_{1},f_{2}$ be two functions as in \eqref{eq:distr},
  and define
  \begin{equation*}
    w(t)=\one{\mathbb{R}^{+}}(t)\cdot 2(\pi t)^{-\frac12}
       (\cos(\sqrt{tH})f_{1},f_{2})_{L^{2}}.
  \end{equation*}
  Notice that $w(t)$ is a tempered distribution on $\mathbb{R}$
  and so are the products $e^{ty}w(t)$ for any $y\le0$. Thus the
  Fourier-Laplace transform
  \begin{equation*}
    v(z)=\int w(t)e^{-izt}dt
  \end{equation*}
  is well defined and analytic on the half complex plane
  $\Im z<0$. Recalling the subordination formula
  \begin{equation*}
    (e^{-sH}f_{1},f_{2})_{L^{2}}=
    \int_{0}^{\infty}(\cos(t \sqrt{H})f_{1},f_{2})_{L^{2}}
      \frac{2}{\sqrt{\pi s}}e^{-\frac{t^{2}}{4s}}dt,
  \end{equation*}
  via the changes of variables $t\to \sqrt{t}$
  and $s\to 1/(4s)$, we see that $v(z)$ can be computed explicitly as
  \begin{equation*}
    v(z)=(iz)^{-\frac12}
      (e^{-\frac{H}{4iz}}f_{1},f_{2})_{L^{2}}.
  \end{equation*}
  Now introduce
  the analytic function
  \begin{equation}\label{eq:Fz}
    F(z)=z^{\frac12}e^{ir^{2}z}v(z)
    \quad\text{on $\Im z<0$}
  \end{equation}
  for some fixed $r$ satifying \eqref{eq:distr}.
  By spectral calculus we have easily the bound
  \begin{equation*}%\label{eq:growth}
    |v(z)|\le |z|^{-\frac12}\|f_{1}\| \cdot\|f_{2}\|
  \end{equation*}
  (all the norms in this proof are $L^{2}$ norms)
  which implies the growth rate
  \begin{equation}\label{eq:growthF}
    |F(z)|\le \|f_{1}\|\cdot\|f_{2}\|\cdot e^{r^{2}|z|}.
  \end{equation}
  If we fix a $y_{0}<0$, again by spectral calculus we obtain
  the bound
  \begin{equation}\label{eq:bound1}
    |F(x+iy_{0})|\le \|f_{1}\|\cdot\|f_{2}\|
  \end{equation}
  along the line $z=x+i y_{0}$, $x\in \mathbb{R}$.
  Finally, along the half line $z=it$, $t<0$, we obtain by
  assumption \eqref{eq:davg}
  \begin{equation}\label{eq:bound2}
    |F(it)|\le C\|f_{1}\|\cdot\|f_{2}\|.
  \end{equation}
  Now we can apply the Phragm\'en-Lindel\"of theorem on the
  two sectors $\Im z\le y_{0}$ and $\Re z\ge0$ or $\Re z\le0$
  (see Theorem IV.3.4 in 
  \cite{SteinShakarchi03-b}) and we obtain
  that $F(z)$ satisfies a bound like \eqref{eq:bound2} on
  the whole half plane $\Im z\le y_{0}$. This implies
  an exponential growth rate
  \begin{equation}\label{eq:boundv}
    |v(z)|\le |z|^{-\frac12} e^{r^{2}\Im z}\|f_{1}\|\cdot\|f_{2}\|,
    \qquad
    \Im z\le y_{0}<0
  \end{equation}
  for the transform of $w(t)$. To conclude the proof, it is
  sufficient to use the Paley-Wiener theorem (see
  Theorem 7.4.3 in 
  \cite{Hormander90-a}) which implies that the
  support of $w(t)$ must be contained in the closed convex set
  \begin{equation}\label{eq:suppw}
    \supp w \subseteq [r^{2},+\infty)
  \end{equation}
  and this gives \eqref{eq:fsp} as claimed.
\end{remark}

\begin{lemma}\label{lem:heb0}
  Assume $H$ satisfies (H) and let
  $g$ be an even function with $\supp g \subseteq[-R,R]$.
  Then we have for all $a\ge0$
  \begin{equation}\label{eq:heb0}
    \|g(\sqrt{H})\|_{\bra{x}^{a}}\le c(n,a,R)\cdot K_{0}
    \|\bra{\xi}^{a+n/2} \widehat{g}\|_{L^{1}}
  \end{equation}
  \begin{equation}\label{eq:heb0new}
    \|\sqrt{H}g(\sqrt{H})\|_{\bra{x}^{a}}\le c(n,a,R)\cdot K_{0}
    \|\bra{\xi}^{a+n/2} \widehat{g}'\|_{L^{1}}
  \end{equation}
  where $c(n,a,R)$ is independent of the operator $H$
  and $K_{0}$ is defined in \eqref{eq:heatest}.
\end{lemma}

\begin{proof}%[of ...]
  It is sufficient to estimate the quantity
  \begin{equation*}
    \sup_{y}\int\left|
      g(\sqrt{H})(x,y)\langle x-y\rangle^{a} 
    \right|dx
  \end{equation*}
  since the symmetric one follows from the same computation
  applied to the adjoint kernel
  $g(\sqrt{H})^{*}(x,y)=\overline{g}(\sqrt{H})(y,x)$.
  Let $G(s)=g(s)e^{s^{2}}$.
  Since $G$ is an an even function, apart from a $(2\pi)^{-1}$
  factor we can write
  \begin{equation*}
    G(t)=\int_{-\infty}^{+\infty}\widehat{G}(\xi)\cos(t \xi)d \xi
  \end{equation*}
  and we have
  \begin{equation*}
    g(\sqrt{H})=
    G(\sqrt{H})e^{-H}=
      \int \widehat{G}(\xi)\cos(\xi \sqrt{H}) e^{-H}d\xi.
  \end{equation*}
  We decompose $G$ using a non homogeneous Paley-Littlewood partition
  of unity $\chi_{j}(\xi)$, $j\ge0$ (the support of $\chi_{j}(s)$
  being $s\sim 2^{j}$) as
  \begin{equation*}
    G=\sum_{j\ge0}G_{j}, \qquad
    \widehat{G_{j}}(s)=\chi_{j}(s)\widehat{G}.
  \end{equation*}
  Then we have to estimate the integrals
  \begin{equation*}
    I_{j}=
    \int|G_{j}(\sqrt{H})e^{-H}(x,y)|\langle x-y \rangle^{a} dx\le
      \int |\widehat{G_{j}}(\xi)|
      \int|\cos(\xi \sqrt{H}) e^{-H}|
         \langle x-y \rangle^{a} dx d\xi.
  \end{equation*}
  The innermost integral can be written in full
  \begin{equation*}
    II=\int \left|
      \int \cos(\xi \sqrt{H})(x,z)e^{-H}(z,y)dz
    \right|\langle x-y\rangle^{a}dx 
  \end{equation*}
  We introduce a a partition of $\mathbb{R}^{n}$ in
  almost disjoint unit cubes $Q$ and denote with $\one{Q}$
  their characteristic functions. Then we can write
  \begin{equation*}
    II\le\sum _{Q}II_{Q},\qquad
    II_{Q}=\int \left|
      \int \cos(\xi \sqrt{H})(x,z)e^{-H}(z,y)\one{Q}(z) dz
    \right|\langle x-y\rangle^{a}dx.
  \end{equation*}
  If $z_{q}$ is the center of the cube $Q$ we have
  \begin{equation*}
    |x-z_{Q}|\lesssim \bra{\xi}
  \end{equation*}
  by the finite speed of propagation for $\cos(\xi \sqrt{H})(x,z)$
  (see Remark \ref{rem:finitesp}),
  and recalling that $\xi\in\supp \widehat{G}_{j}$ we have also
  \begin{equation*}
    \langle x-y\rangle \le
    \langle x-z_{Q}\rangle \langle z_{Q}-y\rangle
    \lesssim
    \bra{\xi} \langle z_{Q}-y\rangle
    \lesssim
    2^{j} \langle z_{Q}-y\rangle.
  \end{equation*}
  Thus by Cauchy-Schwartz in $dx$ we obtain
  \begin{equation*}
    II_{Q}^{2}\lesssim
    \bra{\xi}^{n+2a}\langle z_{Q}-y\rangle^{2a}
    \int
    \left|
      \int \cos(\xi \sqrt{H})(x,z)e^{-H}(z,y)\one{Q}(z) dz
    \right|^{2} dx.
  \end{equation*}
  Using the unitarity of $\cos(\xi \sqrt{H})$ 
  and the gaussian estimate, this gives
  \begin{equation*}
    II_{Q}^{2}\lesssim
    2^{j(n+2a)}\langle z_{Q}-y\rangle^{2a}
    \int
    \left|e^{-H}\one{Q}
    \right|^{2} dz
    \lesssim
    2^{j(n+2a)}K_{0}^{2}
    \int_{Q} e^{-2|z-y|^{2}}\langle z-y\rangle^{2a}
    dz
  \end{equation*}
  and hence, taking square roots and summing over $Q$
  we conclude
  \begin{equation*}
    II \le c(n,a) \cdot 2^{(a+n/2)j} K_{0}
  \end{equation*}
  independently of $y$. Inserting this into $I_{j}$ we see that
  \begin{equation*}
    I_{j}\le c(n,a)K_{0}2^{(a+n/2)j}\int|\widehat{G_{j}}(\xi)|d\xi
      \le c_{1}(n,a)K_{0}\|\bra{\xi}^{a+n/2}\widehat{G_{j}}(\xi)\|_{L^{1}}
  \end{equation*}
  and summing over $j$
  \begin{equation*}
    \|g(\sqrt{H})\|_{\langle x\rangle^{a} }
      \le c(n,a)\|\langle \xi\rangle ^{a+n/2}\widehat{G}(\xi)\|_{L^{1}}.
  \end{equation*}
  Finally we can write
  \begin{equation*}
    G(s)=g(s)e^{s^{2}}=g(s)\cdot \chi(s)e^{s^{2}}
  \end{equation*}
  with $\chi(s)$ a cutoff function equal to 1 on $[-R,R]$.
  Then we have
  \begin{equation*}
    \widehat{G}=\widehat{g}*\widehat{(\chi e^{s^{2}})}\quad \implies
    \quad
    \|\bra{\xi}^{s}\widehat{G}\|_{L^{1}}\le
    c(s,R)\|\bra{\xi}^{s}\widehat{g}(s)\|_{L^{1}}
  \end{equation*}
  whence \eqref{eq:heb0} follows; indeed, the symmetric
  quantity obtained by switching $x,y$ in $I$ is estimated in an
  identical way.
  
  The proof of \eqref{eq:heb0new} is similar: we must estimate now
  the integrals
  \begin{equation*}
    I'_{j}=
    \int|\sqrt{H}G_{j}(\sqrt{H})e^{-H}|\cdot\langle x-y \rangle^{a}dy
      \le
      \iint |\widehat{G_{j}}'(\xi)|\cdot|\cos(\xi \sqrt{H}) e^{-H}|
         \langle x-y \rangle^{a} d\xi dy
  \end{equation*}
  where we used that
  $\widehat{sG(s)}=i\widehat{G}'(\xi)$.
  Proceeding as above we obtain
  \begin{equation*}
    \|\sqrt{H}g(\sqrt{H})\|_{\langle x\rangle^{a} }
      \le c(n,a)\|\langle \xi\rangle ^{a+n/2}\widehat{G}'(\xi)\|_{L^{1}}
  \end{equation*}
  and to conclude it is sufficient to remark that
  \begin{equation*}
    \widehat{G}'=\widehat{g}'*\widehat{(\chi e^{s^{2}})}\quad \implies
    \quad
    \|\bra{\xi}^{s}\widehat{G}\|_{L^{1}}\le
    c(s,R)\|\bra{\xi}^{s}\widehat{g}'\|_{L^{1}}.
  \end{equation*}
\end{proof}

\begin{lemma}\label{lem:heb1}
  Assume $H$ satisfies (H) and $\phi$ is given by
  \eqref{eq:phi}. Let $g$ be
  a function on $\mathbb{R}^{+}$, and define, for $j\in \mathbb{R}$,
  $g_{j}(s)=\phi(2^{j}s)g(s)$. Then for any $a\ge0$
  \begin{equation}\label{eq:heb1b}
    \|g_{j}(\sqrt{H})\|_{\langle 2^{-j}x\rangle ^{a}}\le
     c(n,a)K_{0} \cdot
      \|\bra{\xi}^{a+\frac{n}{2}}\mathcal{F}[\phi(s) S_{2^{-j}}g]\|_{L^{1}},
  \end{equation}
  \begin{equation}\label{eq:hebnew}
    \|\sqrt{H}g_{j}(\sqrt{H})\|_{\langle 2^{-j}x\rangle ^{a}}\le
     c(n,a)K_{0} \cdot
      \|\bra{\xi}^{a+\frac{n}{2}}\mathcal{F}[s\phi(s) S_{2^{-j}}g]\|_{L^{1}}
      \cdot 2^{-j}.
  \end{equation}
\end{lemma}

\begin{proof}%[of ...]
  Extend $g(s)$ for $s\le0$ as an even function; notice that
  the values of $g$ on $(-\infty,0]$ are irrelevant in the
  definition of $g(\sqrt{H})$. We can write
  \begin{equation}\label{eq:rescal}
    g_{j}(\sqrt{H})=S_{2^{-j}}G_{j}(\sqrt{H_{j}})S_{2^{j}}
  \end{equation}
  where
  \begin{equation*}
    G_{j}(s)=\phi(s)g(2^{-j}s)=\phi S_{2^{-j}}g
  \end{equation*}
  and
  \begin{equation*}
    H_{j}=2^{2j}S_{2^{j}}HS_{2^{-j}}.
  \end{equation*}
  It is easy to check by rescaling
  that the operator $H_{j}$ satisfies the 
  conditions in Assumption (H) with the same constants.
  Thus we can apply Lemma \ref{lem:heb0} and obtain
  \begin{equation*}
    \|G_{j}(\sqrt{H_{j}})\|_{\bra{x}^{a}}\le
    c(n,a,R)K_{0}
    \|\bra{\xi}^{a+\frac{n}{2}}\mathcal{F}[\phi S_{2^{-j}}g]\|_{L^{1}}.
  \end{equation*}
  As a consequence of \eqref{eq:rescal},
  the kernels of $G_{j}(\sqrt{H_{j}})$ and $g_{j}(\sqrt{H})$ 
  are related by
  \begin{equation*}
    g_{j}(\sqrt{H})(x,y)=G_{j}(\sqrt{H_{j}})(2^{-j}x,2^{-j}y)\cdot 2^{-jn}.
  \end{equation*}
  and this implies \eqref{eq:heb1b}. Since we have also
  \begin{equation*}
    \sqrt{H_{j}}=2^{j}S_{2^{j}}\sqrt{H} S_{2^{-j}}
  \end{equation*}
  \eqref{eq:hebnew} follows immediately from \eqref{eq:heb0new}.
\end{proof}

\begin{lemma}\label{lem:heb2}
  Assume $H$ satisfies (H), let 
  $\alpha\in C^{\infty}_{c}(\mathbb{R})$ be an even function,
  and for $r>0$ write $\alpha_{r}(s)=\alpha(rs)$. 
  Then, for all $m\ge0$,
  \begin{equation}\label{eq:heb2}
    |\alpha_{r}(\sqrt{H})(x,y)|\le 
       C(n,m,\alpha)K_{0}^{2} \cdot
       \left\langle \frac{x-y}{r}\right\rangle^{-m}r^{-n},
  \end{equation}
  \begin{equation}\label{eq:heb2new}
    |\sqrt{H}\alpha_{r}(\sqrt{H})(x,y)|\le 
       C(n,m,\alpha)K_{0}^{2} \cdot
       \left\langle \frac{x-y}{r}\right\rangle^{-m}r^{-n-1}.
  \end{equation}
\end{lemma}

\begin{proof}%[of ...]
  By rescaling, as in the proof of the previous lemma, we can
  reduce to the case $r=1$. Then define $G(s)=\alpha(s)e^{s^{2}}$ 
  so that, using the inequality
  \begin{equation*}
    \langle x-y\rangle\le \langle x-z\rangle \langle z-y\rangle  ,
  \end{equation*}
  we can write
  \begin{equation*}
    \langle x-y\rangle ^{m}|\alpha(\sqrt{H})(x,y)|\le
    \int |G(\sqrt{H})(x,z)|\langle x-z\rangle^{m} 
    \cdot |e^{-H}(z,y)|\langle z-y\rangle ^{m}
    dz.
  \end{equation*}
  Now we have
  \begin{equation*}
    |p_{1}(z,y)|\cdot\langle z-y\rangle ^{m}\le
    K_{0} \cdot c(n,m)
  \end{equation*}
  and this implies
  \begin{equation*}
    \langle x-y\rangle ^{m}|\psi(\sqrt{H})(x,y)|\le
    c(n,m)K_{0}\|G(\sqrt{H})\|_{\langle x\rangle^{m} }.
  \end{equation*}
  Applying \eqref{eq:heb0} with $a=m$ we obtain
  \begin{equation*}
    \|G(\sqrt{H})\|_{\langle x\rangle^{m} }\le c(n,m,\alpha)K_{0}
  \end{equation*}
  and \eqref{eq:heb2} follows. Analogously, \eqref{eq:heb2new}
  follows from \eqref{eq:hebnew}.
\end{proof}

We can now conclude the proof of \eqref{eq:hebLp}
in a similar way as \cite{Hebisch90-b}.
Let $f\in L^{1}$, $\lambda>0$ and consider the
Calder\`on-Zygmund decomposition of $f$: a sequence
of disjoint cubes
$Q_{j}$ and functions $h,f_{j}$ with $\supp f_{j}\subseteq Q_{j}$,
$j\ge1$, such that
\begin{equation*}
  f=h+\textstyle\sum_{j}f_{j},\qquad |h|\le C \lambda,\qquad
   \textstyle\int|f_{j}|\le C \lambda|Q_{j}|,\qquad
   \sum|Q_{j}|\le C \lambda^{-1}\|f\|_{L^{1}}.
\end{equation*}
Then we can write $g(\sqrt{H})f$ as
\begin{equation}\label{eq:3pieces}
  g(\sqrt{H})f=g(\sqrt{H})h +
     \textstyle\sum_{j}g(\sqrt{H})\psi_{r_{j}}(\sqrt{H})f_{j}+
     \textstyle\sum_{j}(1-\psi_{r_{j}}(\sqrt{H}))f_{j}
\end{equation}
where
\begin{equation*}
  2^{r_{j}}=4\mathop{\mathrm{diam}}(Q_{j}).
\end{equation*}
For the first term in \eqref{eq:3pieces} we have, by the spectral
theorem,
\begin{equation*}
  |\{|g(\sqrt{H})h|>\lambda\}|\le \lambda^{-2}\|g(\sqrt{H})h\|_{L^{2}}^{2}
  \le \lambda^{-2}\|g\|_{L^{\infty}}^{2}\|h\|_{L^{2}}^{2}\le
  C \lambda^{-1}\|g\|_{L^{\infty}}^{2}\|h\|_{L^{1}}
\end{equation*}
and hence
\begin{equation}\label{eq:partI}
  |\{g(\sqrt{H})h>\lambda\}|\le C \|g\|_{L^{\infty}}^{2}\|f\|_{L^{1}}
  \cdot\lambda^{-1}
\end{equation}
since $\|h\|_{L^{1}}\le C\|f\|_{L^{1}}$. To handle the second term, 
we consider the product with $\gamma(x)\in L^{2}$
\begin{equation*}
  |(\psi_{r_{j}}(\sqrt{H})f_{j},\gamma)_{L^{2}}|\le
  CK_{0}^{2}\iint 
  \left\langle \frac{x-y}{r_{j}}\right\rangle^{-m}r_{j}^{-n}
  |\gamma(x)f_{j}(y)|dxdy
\end{equation*}
where we have used estimate \eqref{eq:heb2} for the kernel.
Now we nortice that for all $y\in Q_{j}$ we have
\begin{equation*}
  \left\langle \frac{x-y}{r_{j}}\right\rangle^{-m}\le
  c(m,n)\int_{Q_{j}}
  \left\langle \frac{x-z}{r_{j}}\right\rangle^{-m}dz \cdot|Q_{j}|
\end{equation*}
with a constant independent of $j$. Thus,
using $\int|f_{j}(y)|dy\le C \lambda|Q_{j}|$,
\begin{equation*}
  |(\psi_{r_{j}}(\sqrt{H})f_{j},\gamma)_{L^{2}}|\le
  CK_{0}^{2}\lambda
  \int_{Q_{j}}dz\int
  \left\langle \frac{x-z}{r_{j}}\right\rangle^{-m}r_{j}^{-n}
  |\gamma(x)|dx.
\end{equation*}
The innermost integral is bounded by $c_{n} M \gamma(z)$ provided
we choose e.g. $m=n+1$, so that
\begin{equation*}
  \sum_{j}
  |(\psi_{r_{j}}(\sqrt{H})f_{j},\gamma)_{L^{2}}|\le
  CK_{0}^{2}\lambda \cdot\int_{Q_{j}}M \gamma(z) dz\le
  CK_{0}^{2}\lambda \|M \gamma\|_{L^{2}}
  \|\textstyle\sum \one{Q_{j}}\|_{L^{2}}
\end{equation*}
and noticing that $\|\textstyle\sum \one{Q_{j}}\|_{L^{2}}\le
C\lambda^{-1/2}\|f\|_{L^{1}}^{1/2}$ we find
\begin{equation*}
  \sum_{j}
  |(\psi_{r_{j}}(\sqrt{H})f_{j},\gamma)_{L^{2}}|\le
  CK_{0}^{2}\lambda ^{1/2}\|f\|_{L^{1}}^{1/2}\|\gamma\|_{L^{2}}.
 \end{equation*}
This implies
\begin{equation*}
  \|g(\sqrt{H})
  \sum_{j}\psi_{r_{j}}(\sqrt{H})f_{j}\|_{L^{2}}^{2}\le
  CK_{0}^{4}\|g\|_{L^{\infty}}
  \lambda\|f\|_{L^{1}}
\end{equation*}
and proceeding as for the first piece we obtain
\begin{equation}\label{eq:partII}
  |\{|g(\sqrt{H})
  \sum_{j}\psi_{r_{j}}(\sqrt{H})f_{j}|>\lambda\}|
  \le C K_{0}^{4}\|g\|_{L^{\infty}}^{2}\|f\|_{L^{1}}
  \cdot\lambda^{-1}
\end{equation}
Finally, consider the third piece in \eqref{eq:3pieces}
\begin{equation*}
  III=\sum_{j}(1-\psi_{r_{j}}(\sqrt{H}))f_{j}.
\end{equation*}
Recalling that
\begin{equation*}
  1-\psi(s)=\sum_{k\le0}\phi(2^{k}s)
  \ \text{\ for $s>0$},
\end{equation*}
using the notation $\lg r=\log_{2}r$,
\begin{equation*}
  1-\psi_{r_{j}}(s)=1-\psi(r_{j}s)=
  \sum_{k\le0}\phi(2^{k}r_{j}s)\equiv
  \sum_{k\le0}\phi(2^{k+\lg r_{j}}s)
  \text{\ \ \ for $s>0$}
\end{equation*}
we can write
\begin{equation*}
  III=\sum_{k\le0}g_{k+\lg r_{j}}(\sqrt{H}),\qquad
  g_{j}(s)=g(s)\phi(2^{j}s).
\end{equation*}
Now, if $4Q_{j}$ is a cube with the same center as $Q_{j}$
but with sides multiplied by $4$, and $A=\cup 4Q_{j}$,
\begin{equation*}
  |\{|III|>\lambda\}|\le |A|+
  \lambda^{-1}\sum_{j}\sum_{k\le0}
  \int_{\mathbb{R}^{n}\setminus A}
  |g_{k+\lg r_{j}}(x,y)|\cdot|f_{j}(y)|dy.
\end{equation*}
We shall estimate the kernel of $g_{k+\lg r_{j}}$ as follows:
let $a=\sigma-n/2$ (recall that by assumption
$\mu=\mu_{\sigma}(g)<\infty$ for some $\sigma>n/2$, so
that $a>0$), then we can write
\begin{equation*}
  |g_{k+\lg r_{j}}(x,y)|\le 
  \|g_{k+\lg r_{j}}\|_{\langle x/2^{k}r_{j}\rangle^{a} }\cdot
  \left\langle 
  \frac{x-y}{2^{k}r_{j}}
  \right\rangle ^{-a}\le
  c(n,a)K_{0}\mu \cdot 2^{a(k-j)}
\end{equation*}
where we have used \eqref{eq:heb1b},
and the fact that for $x\not\in A$ and $y\in Q_{j}$ we have
$|x-y|\ge 2^{j}r_{j}$. Notice also that $|A|\le c(n)\sum|Q_{j}|$.
Thus we obtain
\begin{equation*}
  |\{|III|>\lambda\}|\le
  c(n)\lambda^{-1}\|f\|_{L^{1}}+
  c(n,a)K_{0}\mu
  \lambda^{-1}
  \sum_{j}\sum_{k\le0}2^{a(k-j)}\|f_{j}\|_{L^{1}}.
\end{equation*}
Since $a>0$, we can sum over $k\le0$ and we conclude
\begin{equation}\label{eq:partIII}
  |\{|III|>\lambda\}|\le 
   c(n,a)(1+K_{0}\mu)\lambda^{-1}\|f\|_{L^{1}}.
\end{equation}
Summing \eqref{eq:partI}, \eqref{eq:partII} and \eqref{eq:partIII}
we obtain \eqref{eq:hebweak}. 

Estimate \eqref{eq:hebLp} for general $p$ can be obtained 
in a standard way by
real interpolation with the $L^{2}$ trivial estimate and duality.
Notice however
that the constant in the Marcinkiewicz interpolation theorem
diverges at both ends: if $p=(1-\theta)/p_{0}+\theta/p_{1}$
and the linear operator $T$ satisfies weak $L^{p_{j}}$ estimates
with constants $C_{j}$, $j=0,1$, then $T$ satisfies a strong
$L^{p}$ estimate with a norm
\begin{equation*}
  \|T\|_{L^{p}\to L^{p}}\le
  2\left(\frac{p}{p-p_{0}}+\frac{p}{p_{1}-p}\right)^{1/p}
  C_{0}^{1-\theta}C_{1}^{\theta}
\end{equation*}
(see e.g. \cite{Grafakos09-a}). Thus a second (complex) interpolation
step between two strong estimates is necessary in order to
get  \eqref{eq:hebLp}.

The proof of \eqref{eq:hebweaknew} requires
a variant of the Calder\`on-Zygmund decomposition for Sobolev
functions due to Auscher \cite{Auscher07-a}: given $f$ with
$\|\nabla f\|_{L^{1}}<\infty$ and $\lambda>0$, there exists
a sequence of cubes $Q_{j}$ with controlled overlapping (i.e.
$\sum \one{Q_{j}}\le N=N(n)$),
and functions $h,f_{j}$ with $f_{j}\in W^{1}_{0}(Q_{j})$
such that
\begin{equation*}
  f=h+\textstyle\sum_{j}f_{j},\qquad |\nabla h|\le C \lambda,\qquad
   \textstyle\int|\nabla f_{j}|\le C \lambda|Q_{j}|,\qquad
   \sum|Q_{j}|\le C \lambda^{-1}\|\nabla f\|_{L^{1}}.
\end{equation*}
We list the modifications necessary in the preceding proof.
The decomposition is obviously
\begin{equation}\label{eq:3piecesnew}
  \sqrt{H}g(\sqrt{H})f=\sqrt{H}g(\sqrt{H})h +
     \textstyle\sum_{j}\sqrt{H}g(\sqrt{H})\psi_{r_{j}}(\sqrt{H})f_{j}+
     \textstyle\sum_{j}\sqrt{H}(1-\psi_{r_{j}}(\sqrt{H}))f_{j}
\end{equation}
with $r_{j}$ as above. The first piece is estimated using
\eqref{eq:hebassnew} instead of the elementary $L^{2}$ bound,
which gives
\begin{equation*}
  |\{|g(\sqrt{H})h|>\lambda\}|
  \le \lambda^{-q}C_{q}^{q}\|\nabla h\|_{L^{q}}^{q}\le
  CC_{q}^{q} \lambda^{-1}\|\nabla h\|_{L^{1}}\le
  CC_{q}^{q} \lambda^{-1}\|\nabla f\|_{L^{1}}.
\end{equation*}
For the second piece we write as before, but using now the kernel
estimate \eqref{eq:heb2new},
\begin{equation*}
  |(\sqrt{H}\psi_{r_{j}}(\sqrt{H})f_{j},\gamma)_{L^{2}}|\le
  CK_{0}^{2}\iint 
  \left\langle \frac{x-y}{r_{j}}\right\rangle^{-m}r_{j}^{-n-1}
  |\gamma(x)f_{j}(y)|dxdy.
\end{equation*}
Notice that Poincar\'e's inequality implies
\begin{equation*}
  \int|f_{j}(y)|dy\le C r_{j}\int|\nabla f_{j}\|dy\le
  C r_{j}\lambda |Q_{j}|
\end{equation*}
and the factor $r_{j}$ cancels the additional power in $r_{j}^{-n-1}$.
Thus we arrive at
\begin{equation*}
  \sum_{j}
  |(\sqrt{H}\psi_{r_{j}}(\sqrt{H})f_{j},\gamma)_{L^{2}}|\le
  CK_{0}^{2}\lambda \cdot\int_{Q_{j}}M \gamma(z) dz
\end{equation*}
and as above this implies
\begin{equation}\label{eq:partIInew}
  |\{|\sqrt{H}g(\sqrt{H})
  \sum_{j}\psi_{r_{j}}(\sqrt{H})f_{j}|>\lambda\}|
  \le C K_{0}^{4}\|g\|_{L^{\infty}}^{2}\|\nabla f\|_{L^{1}}
  \cdot\lambda^{-1}.
\end{equation}
The third piece is decomposed again as
\begin{equation*}
  III'=\sum_{k\le0}\sqrt{H}g_{k+\lg r_{j}}(\sqrt{H}),\qquad
  g_{j}(s)=g(s)\phi(2^{j}s).
\end{equation*}
Using the kernel estimate \eqref{eq:hebnew} we get now,
with $a=\sigma-n/2$ (so that $a>1$ now)
\begin{equation*}
  |\{|III'|>\lambda\}|\le
  c(n)\lambda^{-1}\|\nabla f\|_{L^{1}}+
  c(n,a)K_{0}\mu'
  \lambda^{-1}
  \sum_{j}\sum_{k\le0}2^{a(k-j)}\|f_{j}\|_{L^{1}}\cdot
  2^{-k}r_{j}^{-1}.
\end{equation*}
Since $a>1$ the sum in $k$ converges with sum bounded by
a constant $c(a)$, and another application of
Poincar\'e's inequality cancels the power $r_{j}^{-1}$.
In conclusion
\begin{equation*}
  |\{|III'|>\lambda\}|\le c(n,a)(1+K_{0}\mu')
  \lambda^{-1}\|\nabla f\|_{L^{1}}
\end{equation*}
and the proof is complete.

% section estimates_for_the_kernel (end)

\section{Bounded functions of the operator: Theorem \ref{the:1}}
\label{sec:bddfunct}  %(fold)

\subsection{The Auscher-Martell maximal lemma}\label{sub:AMmaxlemma}  %(fold)
We reproduce here the
maximal lemma of \cite{AuscherMartell07-a},
in a version slightly simplified for our needs
(i.e., in the original Lemma a finer decomposition in condition
\eqref{eq:cond2} is permitted). 
We decided to include a short but complete proof in the
Appendix, since we needed to keep track precisely
of the constants appearing in the final estimate
\eqref{eq:1lemAM}; this gives the additional bonus of making the
paper self-contained. We also took the liberty
of introducing some minor
simplifications in the final step of the proof.

In the statement of Lemma \ref{lem:auscapp} below,
the quantity $a^{q}/K^{q}$ in \eqref{eq:lemAM} 
must be interpreted as $0$ when $q=\infty$, $MF$ denotes the
uncentered maximal operator over balls $B$
\begin{equation}\label{eq:maxunc}
  Mf(x)=\sup_{B\ni x} \dashint_{B}|f(x)|dx,
\end{equation}
and $c_{q}$ is its norm in the weak $(q,q)$ bound
\begin{equation}\label{eq:Mqqbound}
  \sup_{\lambda>0}\lambda^{q}
  |\{Mf>\lambda\}|\le
  c_{q}\|f\|_{L^{q}}^{q},\qquad 1\le q<\infty,\qquad
  c_{\infty}\equiv1.
\end{equation}
We also recall that a weight $w(x)>0$ belongs
the \emph{reverse H\"older class}
$RH_{q}$, $1<q<\infty$, if there exists a constant $C$ such that
for every cube $Q$
\begin{equation}\label{eq:RHq}
  \left(\dashint_{Q}w^q\right)^{1/q}\le C\dashint_{Q}wdx.
\end{equation}
while $RH_{\infty}$ is defined by the condition
\begin{equation}\label{eq:RHi}
  w(x)\le C\dashint_{Q}wdx \quad\text{for a.e. $x\in Q$}.
\end{equation}
The best constant $C$ in these inequalities is denoted by
$\|w\|_{RH_{q}}$.
We shall use the following consequence of the previous definition:
if $w\in RH_{s'}$ for some $1\le s< \infty$, then there exists
$C$ such that for every cube $Q$
and every measurable subset $E \subseteq Q$
\begin{equation}\label{eq:RHineq}
  \frac{w(E)}{w(Q)}\le \|w\|_{RH_{s'}}
    \left(\frac{|E|}{|Q|}\right)^{\frac1s}
\end{equation}
Indeed, for $s'<\infty$ one can write
\begin{equation*}
  \frac{w(E)}{w(Q)}\le
  \frac{|Q|}{w(Q)}\left(\dashint_{Q}w^{s'}\right)^{\frac{1}{s'}}
    \left(\frac{|E|}{|Q|}\right)^{\frac1s}\le
  \|w\|_{RH_{s'}}\left(\frac{|E|}{|Q|}\right)^{\frac1s}
\end{equation*}
while for $s'=\infty$ the proof is even more elementary.

\begin{lemma}[\cite{AuscherMartell07-a}]
  \label{lem:auscapp}
Let $F,G$ be positive measurable functions on $\mathbb{R}^n$, 
$1<q\leq\infty$, $a\geq1$, $1\leq s<\infty$, $w\in RH_{s'}$. 
Assume that for every ball $B$ there exist $G_B$, $H_B$ positive 
functions such that
\begin{equation}\label{eq:cond1}
  F\leq G_B+H_B\quad \ \text{a.e. on $B$,}
\end{equation}
\begin{equation}\label{eq:cond2}
  \|H_{B}\|_{L^{q}(B)}\le a(MF(x)+G(y))\cdot|B|^{\frac1q}
  \quad\text{for every $x,y\in B$,}
\end{equation}
\begin{equation}\label{eq:cond3}
  \|G_{B}\|_{L^{1}(B)}\le G(x)\cdot|B| \quad\text{for every $x\in B$}.
\end{equation}
Then for all $\lambda>0$, $0<\gamma<1$, $K\ge 2^{n+2}a$, we have,
with $C_{0}=2^{6(n+q)}(c_{1}+c_{q})$,
\begin{equation}\label{eq:lemAM}
  w\{MF>K \lambda,\ G\le \gamma \lambda\}\le
  C_{0}\|w\|_{RH_{s'}}\cdot
  \left(
    \frac{\gamma}{K}+\frac{a^{q}}{K^{q}}
  \right)^{\frac1s}
  \cdot
  w\{MF>\lambda\}.
\end{equation}
As a consequence, if $F$ is $L^{1}$ and $1\le p<q/s$,
\begin{equation}\label{eq:1lemAM}
  \|MF\|_{L^{p}(w)}\le C_{1} \|G\|_{L^{p}(w)},\qquad
  C_{1}=
  \left[(8C_{0}\|w\|_{RH_{s'}}+2^{n+3})a^{p}\right]^{\frac{s}{1-ps/q}}.
\end{equation}
\end{lemma}

% subsection auscher martell the1 (end)

\subsection{Proof of Theorem \ref{the:1}}\label{sub:the1}  %(fold)
Assume for the moment $w\in RH_{s'}$ for some $1\le s< \infty$; 
at the end of the proof we shall optimize the choice in order 
to handle a generic weight in $A_{r}$. Moreover, fix a
$\nu>1$ so large that $\sigma>n/\nu$ i.e. $\nu>n/\sigma$.

Given any test function $f$, set $F(x)=|g(\sqrt{H})f|^{\nu}$,
which is in $L^{1}$
by Theorem \ref{the:hebweakest}. Then, for any ball $B$ define,
with $\psi_{r}(s)=\psi(rs)$,
\begin{equation*}
  G_{B}=2^{\nu}|g(\sqrt{H})(1-\psi_{r}(\sqrt{H}))f|^{\nu},\quad
  H_{B}=2^{\nu}|g(\sqrt{H})\psi_{r}(\sqrt{H})f|^{\nu}
\end{equation*}
where $r$ is the radius of the ball $B$.
We will show now that with these choices the assumptions of the 
maximal lemma are satisfied. Clearly we have $F\le G_{B}+H_{B}$ 
a.e.~on $\mathbb{R}^{n}$.

We check that assumption \eqref{eq:cond2} 
holds with $q=\infty$. 
For any $z\in B$ we have, writing for short $T=g(\sqrt{H})$,
\begin{equation*}
  |T\psi_{r}(\sqrt{H})f(z)|\le
    \int |\psi_{r}(\sqrt{H})(z,y)|\cdot|Tf(y)|dy=I.
 \end{equation*}
We can apply Lemma \ref{lem:heb2} with $m=n+1$;
writing $B_{j}=2^{j}B$, $j\ge0$, $B_{-1}=\emptyset$,
we have
\begin{equation*}
  I\le C(n,\psi)K_{0}^{2}
     r^{-n}\sum_{j\ge0}
     \int_{B_{j}\setminus B_{j-1}}
     \left\langle \frac{z-y}r \right\rangle ^{-n-1}|Tf(y)|dy
\end{equation*}
and using $\langle |z-y|/r\rangle \ge 2^{j-1}$ and
$|B_{j}|=2^{nj}r^{n}\omega_{n}$, we obtain
\begin{equation*}
  I\le C(n,\psi)K_{0}^{2}2^{n+1}\omega_{n}
     \sum_{j\ge0}2^{-j}
     \dashint_{B_{j}}|Tf(y)|dy.
\end{equation*}
Now if $x\in B$ and $B'=B(x,r)$, $B'_{j}=2^{j}B'$, we have
\begin{equation*}
  \dashint_{B_{j}}|Tf(y)|dy\le
  c(n) \left(\dashint_{B'_{j+1}}|Tf(y)|^{\nu}dy\right)^{\frac1\nu}\le
  c(n) \cdot MF(x)^{1/\nu}
\end{equation*}
and we obtain \eqref{eq:cond2} with $q=\infty$:
\begin{equation}\label{eq:cond2b}
  |H_{B}(z)|=
    2^{\nu}|T\psi_{r}(\sqrt{H})f(z)|^{\nu}\le
    aMF(x),\qquad
    a=c(n,\psi,\nu)K_{0}^{2\nu}.
\end{equation}

Consider now the remaining term, which we split as
\begin{equation*}
  G_{B}=2^{\nu}
  |g(\sqrt{H})(1-\psi_{r}(\sqrt{H}))f|^{\nu}\le 4^{\nu}(II^{\nu}+III^{\nu})
\end{equation*}
where
\begin{equation*}
  II= |g(\sqrt{H})(1-\psi_{r}(\sqrt{H}))f_{1}|,\qquad
  III=|g(\sqrt{H})(1-\psi_{r}(\sqrt{H}))f_{2}|,
\end{equation*}
\begin{equation*}
  f_{1}=f \cdot\one{4B},\qquad 
  f_{2}=f \cdot\one{\mathbb{R}^{n}\setminus 4B}.
\end{equation*}
For the piece $II$ we use Theorem \ref{the:hebweakest}
(recall that we can take $\nu>>1$):
\begin{equation*}
  \|II\|_{L^{\nu}(B)}\le \nu \cdot
  c(n,\sigma)K_{0}^{4}(1+\mu+\|g\|_{L^{\infty}}^{2})
  \|(1-\psi_{r}(\sqrt{H}))f_{1}\|_{L^{\nu}}.
\end{equation*}
Notice that
\begin{equation*}
  \|(1-\psi_{r}(\sqrt{H}))f_{1}\|_{L^{\nu}}\le
  \|\psi_{r}(\sqrt{H})f_{1}\|_{L^{\nu}}
  +\|f_{1}\|_{L^{\nu}}
\end{equation*}
and using \eqref{eq:heb2} with $m=n+1$ we see that
\begin{equation*}
  \|\psi_{r}(\sqrt{H})f_{1}\|_{L^{\nu}}\le 
  c(n,\psi) K_{0}^{2}\|f_{1}\|_{L^{\nu}}
\end{equation*}
which implies
\begin{equation*}
  \|II\|_{L^{\nu}(B)}\le cK_{0}^{6}
    (1+\mu+\|g\|_{L^{\infty}}^{2})\|f_{1}\|_{L^{\nu}}.
\end{equation*}
Estimating with the maximal function we obtain
\begin{equation}\label{eq:pieceII}
  \|II\|_{L^{\nu}(B)}\le c(n,\sigma,\psi)K_{0}^{6}
    (1+\mu+\|g\|_{L^{\infty}}^{2})
  \cdot r^{n/\nu}\cdot M(|f|^{\nu})(x)^{1/\nu}\qquad
  \forall x\in B.
\end{equation}
We can now focus on the piece $III$; we write
\begin{equation*}
  1-\psi(s)= \sum_{k\le0}\phi(2^{k}s) \text{\ \ \ for $s>0$}
\end{equation*}
and hence, using the notation $\lg r=\log_{2}r$,
\begin{equation*}
  1-\psi_{r}(s)=1-\psi(rs)=
  \sum_{k\le0}\phi(2^{k}rs)\equiv
  \sum_{k\le0}\phi(2^{k+\lg r}s)
  \text{\ \ \ for $s>0$}
\end{equation*}
which implies
\begin{equation*}
  g(\sqrt{H})(1-\psi_{r}(\sqrt{H}))=\sum_{k\le0}g_{k+\lg r}(\sqrt{H}),\qquad
  g_{j}(s)=g(s)\phi(2^{j}s).
\end{equation*}
Denote by $a_{k}(x,y)$ the kernel of $g_{k+\lg r}(\sqrt{H})$, then we have
($B_{j}=2^{j}B$)
\begin{equation*}
  \|g_{k+\lg r}(\sqrt{H})f_{2}\|_{L^{2}(B)}\le
  \sum_{j\ge3}
  \left\|
  \int_{B_{j}\setminus B_{j-1}}|a_{k}(z,y)f_{2}(y)|dy
  \right\|_{L^{2}_{z}(B)}.
\end{equation*}
Now by H\"older's inequality
\begin{equation*}
  \left\|\int_{A}|a(z,y)f(y)|dy\right\|_{L^{\nu}_{z}(B)}\le
  C
  \|f\|_{L^{\nu}(A)}
\end{equation*}
where
\begin{equation}\label{eq:schur}
  C=\max\left\{
  \sup_{z\in A}\left(\int_{B}|a(z,y)|dy\right),
  \sup_{z\in B}\left(\int_{A}|a(z,y)|dy\right)
  \right\}.
\end{equation}
Moreover, Lemma \ref{lem:heb1} and assumption \eqref{eq:assbddg}
ensure that
\begin{equation}\label{eq:estak}
  \|a_{k}\|_{\langle 2^{k}r^{-1}x\rangle^{\sigma}}\le
  c(n,\sigma)K_{0}\mu.
\end{equation}
We notice that for $z\in B$ and $y\in B_{j}\setminus B_{j-1}$,
$j\ge2$, $k\le0$, one has
\begin{equation*}
  \frac{|z-y|}{2^{k}r}\ge 2^{j-k-2}\ge1\ \implies\ 
  \left\langle \frac{z-y}{2^{k}r}\right\rangle ^{\sigma}\ge
  4^{-\sigma}2^{\sigma(j-k)}
\end{equation*}
which together with \eqref{eq:estak} implies for \eqref{eq:schur}
\begin{equation*}
  C
  \le
  c(n,\sigma)K_{0}\mu \cdot 2^{\sigma(k-j)}
\end{equation*}
amd hence
\begin{equation*}
  \left\|
  \int_{B_{j}\setminus B_{j-1}}|a_{k}(z,y)f_{2}(y)|dy
  \right\|_{L^{\nu}_{z}(B)}\le c(n,\sigma)K_{0}\mu \cdot 2^{\sigma(k-j)}
  \|f\|_{L^{\nu}(B_{j}\setminus B_{j-1})}.
\end{equation*}
Now let $x\in B$ arbitrary and $B'=B(x,r)$, $B'_{j}=2^{j}B$, then
\begin{equation*}
  \|f\|_{L^{\nu}(B_{j}\setminus B_{j-1})}\le
  \|f\|_{L^{\nu}(B'_{j+1})}\le
  c_{n}2^{nj/\nu}r^{n/\nu}\cdot M(|f|^{\nu})(x)^{1/\nu},
\end{equation*}
thus we have proved for all $x\in B$
\begin{equation*}
  \left\|
  \int_{B_{j}\setminus B_{j-1}}|a_{k}(z,y)f_{2}(y)|dy
  \right\|_{L^{\nu}_{z}(B)}\le
  c(n,\sigma)K_{0}\mu \cdot 2^{\sigma(k-j)}2^{nj/\nu}r^{n/\nu}
  M(|f|^{\nu})(x)^{1/\nu}.
\end{equation*}
Summing over $j\ge3$, since $\sigma>n/\nu$ we get
\begin{equation}\label{eq:estgk}
  \|g_{k+\lg r}(\sqrt{H})f_{2}\|_{L^{2}(B)}\le
  c(n,\sigma)K_{0}\mu \cdot 2^{k\sigma}r^{n/\nu}\cdot
  M(|f|^{\nu})(x)^{1/\nu}.
\end{equation}
and summing over $k\le0$, and recalling \eqref{eq:pieceII}, we conclude
\begin{equation}\label{eq:estGB}
  \begin{split}
    \|G_{B}\|_{L^{1}(B)}\le & 
    4^{\nu}\|II\|^{\nu}_{L^{\nu}(B)}+
    4^{\nu}\|III\|^{\nu}_{L^{\nu}(B)}
  \\
     \le & \nu^{\nu}  c(n,\sigma)^{\nu}
     K_{0}^{\nu}(1+\mu+\|g\|_{L^{\infty}}^{2})^{\nu} 
     \cdot M(|f|^{\nu})(x)\cdot|B|.
  \end{split}
\end{equation}
This proves \eqref{eq:cond3} with the choice
\begin{equation}\label{eq:choiceG}
  G(x)=\nu^{\nu} 
  c(n,\sigma)^{\nu}
  K_{0}^{\nu}(1+\mu+\|g\|_{L^{\infty}}^{2})^{\nu}  
  \cdot M(|f|^{\nu})(x)
\end{equation}

We are finally in position to apply Lemma \ref{lem:auscapp} and
we obtain, for all $1\le p<\infty$, and any weight
$w\in RH_{s'}$ for some $1\le s<\infty$,
\begin{equation}\label{eq:final}
  \|F\|_{L^{p}(w)}\le\|MF\|_{L^{p}(w)}\le
  C_{1}
  \|G\|_{L^{p}(w)}
\end{equation}
where in our case
\begin{equation*}
  C_{1}=c(n,\sigma,\psi,p,s)(\|w\|_{RH_{s'}}+1)^{s}
  K_{0}^{2ps\nu},
\end{equation*}
that is to say
\begin{equation}\label{eq:interm}
  \|g(\sqrt{H})f\|^{\nu}_{L^{p\nu}(w)}\le C_{2}
  \|M(|f|^{\nu})\|_{L^{p}(w)}
\end{equation}
where
\begin{equation*}
  C_{2}=\nu^{\nu}c(n,\sigma,\psi,p,s)^{\nu}(\|w\|_{RH_{s'}}+1)^{s}
  K_{0}^{\nu+2ps\nu}(1+\mu+\|g\|_{L^{\infty}}^{2})^{\nu}
\end{equation*}

Now, assume the weight is in some $A_{p}$; 
recalling that $\cup_{1\le p<\infty}A_{p}=\cup_{1<q\le \infty}RH_{q}$,
we have also $w\in RH_{s'}$ for some $1\le s<\infty$, and
all the previous computations apply. Since the maximal
operator is bounded on $L^{p}(w)$, we deduce from \eqref{eq:interm}
\begin{equation*}
  \|g(\sqrt{H})f\|_{L^{p\nu}(w)}\le C_{3}
  \|f\|_{L^{p\nu}(w)}
\end{equation*}
where
\begin{equation*}
  C_{3}=\nu \cdot
  c(n,\sigma,\psi,p,w)K_{0}^{1+2p^{2}}
  (1+\mu+\|g\|_{L^{\infty}}^{2}).
\end{equation*}
Let $q=\nu p$; since we can take $\nu>n/\sigma$ 
(provided $\nu>1$) arbitrarily large,
we see that we have proved \eqref{eq:FH} for all 
$q>\max\{p,pn/\sigma\}$,
with a constant
\begin{equation*}
  \frac qp \cdot
  c(n,\sigma,\psi,p,w)K_{0}^{1+2p^{2}}
  (1+\mu+\|g\|_{L^{\infty}}^{2})=
  c'(n,\sigma,\psi,p,w)K_{0}^{1+2p^{2}}
  (1+\mu+\|g\|_{L^{\infty}}^{2})q
\end{equation*}
as claimed.

% subsection proof the1 (end)

% section bdd fun op (end)

\section{The electromagnetic laplacian}\label{sec:electromagnetic}  %(fold)

In this section we verify that an electromagnetic Laplacian
\begin{equation*}
  H=(i\nabla-A(x))^{2}+V(x)
\end{equation*}
satisfies Assumption (H), under
suitable (very weak) regularity and integrability conditions
on the coefficients. 
We recall that a measurable function $V$ on $\mathbb{R}^{n}$
is in the \emph{Kato class} when
\begin{equation*}
  \sup_{x}\lim_{r \downarrow0}
  \int_{|x-y|<r} \frac{|V(y)|}{|x-y|^{n-2}}dy,\qquad (n\ge3)
\end{equation*}
while the \emph{Kato norm} is defined by
\begin{equation*}
  \|V\|_{K}=\sup_{x}\int \frac{|V(y)|}{|x-y|^{n-2}}dy \qquad
  (n\ge3)
\end{equation*}
(replace $|x-y|^{2-n}$ with $\log|x-y|$ in dimension $n=2$).

Our conditions will be based on 
the following result, which is obtained by combining an
heat kernel estimate from \cite{DanconaPierfelice05-a}
with Simon's diamagnetic inequality:

\begin{proposition}\label{pro:heatk}
  Consider the Schr\"odinger operator $H=(i\nabla-A(x))^{2}+V(x)$
  on $L^{2}(\mathbb{R}^{n})$, $n\ge3$.
  Assume that $A\in L^{2}_{loc}(\mathbb{R}^n, \mathbb{R}^n)$, 
  moreover the positive
  and negative parts $V_{\pm}$ of $V$ satisfy
  \begin{equation}\label{eq:Vpiu}
    V_{+}\ \text{is of Kato class},
  \end{equation}
  \begin{equation}\label{eq:Vmeno}
    \|V_{-}\|_K< c_{n}=\pi^{n/2}/\Gamma\left(n/2-1\right).
  \end{equation}
  Then $H$ has a unioque nonnegative selfadjoint extension,
  $e^{-tH}$ is an integral operator whose kernel
  satisfies the pointwise estimate
  \begin{equation}\label{eq:heatest}
    |e^{-tH}(x,y)|\le \frac{K_{0}}{t^{n/2}}
      e^{-|x-y|^{2}/(8t)},\qquad
      K_{0}=\frac{(2\pi )^{-n/2}}{1-\|V_{-}\|_{K}/c_{n}}.
  \end{equation}
\end{proposition}

\begin{proof}%[of ...]
  Simon's diamagnetic pontwise inequality
  (see Theorem B.13.2 in \cite{Simon82-a}), which holds under weaker
  assumptions, states that for any test function $\phi(x)$,
  \begin{equation*}
    |e^{t[(\nabla-iA(x))^{2}-V]}\phi|\le
    e^{t (\Delta-V)}|\phi|.
  \end{equation*}
  By choosing a delta sequence $\phi_{\epsilon}$ of
  test functions, this implies an analogous pointwise inequality
  for the corresponding heat kernels. Now we can apply the second
  part of Proposition 5.1 in \cite{DanconaPierfelice05-a} 
  which gives precisely
  estimate \eqref{eq:heatest} for the heat kernel of
  $e^{-t (\Delta-V)}$ under \eqref{eq:Vpiu}, \eqref{eq:Vmeno}.
\end{proof}

% section electromagnetic (end)

\section{Fractional powers: proof of Corollary \ref{the:2}}\label{sec:th2}  %(fold)

Theorem \ref{the:2} will be proved via
Stein-Weiss interpolation for a suitable analytic family of operators
We need the following lemma:

\begin{lemma}\label{lem:lap}
  Assume $n\ge3$, $1<p<n/2$, and let $w(x)$ be a weight of class $A_{p}$.
  Then the operator $H=(i \nabla-A)^{2}+V$ satisfies the estimate
  \begin{equation}\label{H-lapeq}
    \| Hg\|_{L^p(w)}\leq
    c(n,p,w) \cdot(\||A|^2-i\nabla\cdot A+V\|_{L^{n/2}}+
    \| A\|_{L^n}+1)
    \| (-\Delta)g\|_{L^p(w)}
  \end{equation}
\end{lemma}

\begin{proof}
Setting $w=v^p$, the right hand side of \eqref{H-lapeq} 
can be written $\|vHg\|_{L^{p}}$.
If we expand the operator $H$ and use H\"older's inequality 
for Lorentz spaces we find
\begin{equation*}
  \|vHg\| _{L^p}\le
   \||A|^2-i\nabla\cdot A+V\|_{L^{n/2,\infty}}\|v g\|_{L^{p^{**},p}}+
       2\|A\|_{L^{n,\infty}}\|v \nabla g\|_{L^{p^{*},p}}
\end{equation*}
where
\begin{equation*}
  p^{*}=\frac{np}{n-p},\qquad
  p^{**}=\frac{np}{n-2p}.
\end{equation*}
We can use now the weighted version of Sobolev embeddings proved
by Muckenhopt and Wheeden (see \cite{MuckenhouptWheeden74-a}
and \cite{AuscherMartell08-a}).
Recall also the definition of the reverse H\"older class
\eqref{eq:RHq} -- \eqref{eq:RHi}.

\begin{theorem}\label{MW}
  For $1<p\le q<\infty$ we have
  $$
  \|v (-\Delta)^{-\alpha/2}g\|_{L^q}\leq C\|v g\|_{L^p}
  $$
  provided $\displaystyle\frac{\alpha}{n}=\frac{1}{p}-\frac{1}{q}$ and
  $v\in A_{2-\frac{1}{p}}\cap RH_q$.
\end{theorem}

By real interpolation the preceding estimates extend easily
to Lorentz spaces as follows
\begin{equation}\label{eq:lor}
  \|v (-\Delta)^{-\alpha/2}g\|_{L^{q,p}}\leq C\|v g\|_{L^p},
\end{equation}
under the same conditions on $p,q,w$.
Notice that this result for $\alpha=1,2$, combined with 
the boundedness of the Riesz operator $\nabla(-\Delta)^{-1/2}$ 
in weighted spaces, gives precisely the estimates we need:
\begin{equation*}
  \|v g\|_{L^{p^{**},p}}\le C\|v(-\Delta)g\|_{L^{p}},\qquad
  \|v \nabla g\|_{L^{p^{*,p}}}\le C\|v(-\Delta)g\|_{L^{p}}
\end{equation*}
as soon as the weights are in the appropriate classes.
In order to apply Theorem \ref{MW} we must require that
\begin{equation*}
  v=w^{1/p}\in 
    A_{2-\frac{1}{p}}\cap RH_\frac{np}{n-p}\cap RH_\frac{np}{n-2p}
\end{equation*}
We now use a few basic properties of weighted spaces and 
reverse H\"{o}lder classes (for more details see 
\cite{Garcia-CuervaRubio-de-Francia85-a}).
First of all, for $1\leq r\leq\infty$ and $1<q<\infty$ one has
$$
v\in A_r\cap RH_q\Leftrightarrow v^q\in A_{q(r-1)+1}.
$$
Setting $q=p=q(r-1)+1$, which implies $r=2-1/p$, we obtain 
$$
v\in A_{2-\frac{1}{p}}\cap RH_p\Leftrightarrow w=v^p\in A_p.
$$
Since the classes $RH_q$ are decreasing in $q$, i.e.
$$
RH_\infty\subset RH_q\subset RH_p,\quad\textrm{for}\:1<p\leq q\leq\infty
$$
and $p<p^{*}<p^{**}$, all conditions on $v$ collapse to $w\in A_p$
and the proof is concluded.
\end{proof}

Now fix $1<p_{0}<\infty$, $1<p_{1}<n/2$, and two weights
$w_{0}\in A_{p_{0}}$,
$w_{1}\in A_{p_{1}}$, and consider the family of operators
for $z$ in the strip $0\le \Re z\le1$
\begin{equation*}
  T_{z}=w_{z}H^{z}(-\Delta)^{-z}w_{z}^{-1},
  \qquad
  w_{z}^{\frac{1}{p_{z}}}=
    w_{0}^{\frac{1-z}{p_{0}}}w_{1}^{\frac{z}{p_{1}}},
  \qquad
  \frac{1}{p_{z}}=\frac{1-z}{p_{0}}+\frac{z}{p_{1}}.
\end{equation*}
We follow here the standard theory of
\cite{SteinWeiss71-a} (see Theorem V.4.1),
and in particular
the operators $T_{z}$ are defined on simple functions 
$\phi$ belonging to
$L^{1}(\mathbb{R}^{n})$, with values into measurable functions.
Moreover, we have
\begin{equation*}
  |T_{1+iy}\phi|=w_{1}^{\frac{1}{p_{1}}}|H^{iy}H(-\Delta)(-\Delta)^{-iy}
     w_{1}^{-\frac{1}{p_{1}}}
     (w_{0}^{1/p_{0}}w_{1}^{-1/p_{1}})^{iy}
     \phi|.
\end{equation*}
The function $g(s)=s^{2iy}$ satisfies 
$\mu_{\sigma}(g)\le C(1+|y|)^{\sigma}< \infty$ for all $\sigma$
(see Remark \ref{rem:compareheb}), so choosing e.g.~$\sigma=n+1$,
by the weighted estimate \eqref{eq:FH} we have that $H^{iy}$ is
bounded on $L^{q}(w)$ for all $w\in A_{p}$ and all
$q\ge p$ (actually $q>p-\epsilon$ as per Remark \ref{rem:Ap}).
This applies also to the special case of the operator $(-\Delta)^{iy}$.
Combining \eqref{eq:FH} with 
Lemma \ref{lem:lap}, we deduce
\begin{equation*}
  \|T_{1+iy}\phi\|_{L^{p_{1}}}\le 
  c(n,p_{1},w_{1})K_{0}^{1+2p_{1}^{2}}
  C(A,V)
  (1+|y|)^{n+1}
  \|\phi\|_{L^{p_{1}}},
\end{equation*}
where
\begin{equation}\label{eq:CAV}
  C(A,V)=\||A|^2-i\nabla\cdot A+V\|_{L^{n/2}}+
  \| A\|_{L^n}+1.
\end{equation}
Notice in particular the polynomial growth in $y$
which ensures that $T_{z}$ is an admissible family
in the sense of \cite{SteinWeiss71-a}.
On the other hand we have
\begin{equation*}
  |T_{iy}\phi|=w_{0}^{\frac{1}{p_{0}}}|H^{iy}(-\Delta)^{-iy}
     w_{0}^{-\frac{1}{p_{0}}}
     (w_{0}^{1/p_{0}}w_{1}^{-1/p_{1}})^{iy}
     \phi|
\end{equation*}
and by a similar argument we deduce
\begin{equation*}
  \|T_{iy}\phi\|_{L^{p_{0}}}\le 
  c(n,\epsilon,p_{0},w_{0})K_{0}^{1+2p_{0}^{2}}
  (1+|y|)^{n}\|\phi\|_{L^{p_{0}}}.
\end{equation*}
Thus we are in position to apply complex
interpolation for the family $T_{z}$,
and we conclude that, for $0<\theta<1$,
\begin{equation*}
  \|T_{\theta}\phi\|_{L^{p_{\theta}}}\le 
  c(n,p_{j},w_{j})K_{0}^{2(1+p_{0}^{2}+p_{1}^{2})}
  C(A,V)^{\theta}
  \|\phi\|_{L^{p_{\theta}}}
\end{equation*}
which is equivalent to
\begin{equation*}
  \|H^{\theta}\phi\|_{L^{p_{\theta}}(w_{\theta})}\le
  c(n,\epsilon,p_{j},w_{j})K_{0}^{2(1+p_{0}^{2}+p_{1}^{2})}
  C(A,V)^{\theta}
  \|(-\Delta)^{\theta}\phi\|_{L^{p_{\theta}}(w_{\theta})}.
\end{equation*}
Notice that
\begin{equation}\label{eq:interpp}
  \frac{1}{p_{\theta}}=\frac{1-\theta}{p_{0}}+\frac{\theta}{p_{1}}
\end{equation}
and since $1<p_{0}<\infty$, $1<p_{1}<n/2$ are arbitrary,
$p_{\theta}$ can be any index in the range $1<p<n/(2 \theta)$.

Summing up, we have proved inequality \eqref{eq:fract}
for all choices of $0<\theta<1$, $1<p<n/(2 \theta)$
and all weights $w(x)$ which can be represented in the form
\begin{equation}\label{eq:choicew}
  w=w_{0}^{p_{\theta}\frac{1-\theta}{p_{0}}}
     w_{1}^{p_{\theta}\frac{\theta}{p_{1}}},
\end{equation}
with $w_{j}\in A_{p_{j}}$.
The indices $p_{0},p_{1}$ must be such that
\begin{equation*}
  \frac{1}{p}=\frac{1-\theta}{p_{0}}+\frac{\theta}{p_{1}}
\end{equation*}
and of course $1<p_{0}<\infty$, $1<p_{1}<n/2$. 
It is clear that the weights of the form \eqref{eq:choicew}
belong to $A_{p}$ (using e.g.~the characterization in therms
of maximal estimates). Conversely, it is not difficult
to see that any $A_{p}$ weight can be represented in the
form \eqref{eq:choicew}. Indeed, recall the following characterization
of Muckenhoupt weights (see \cite{Stein93-a}): $w\in A_{p}$,
$1\le p<\infty$, if and only if there exist two weights 
$a(x),b(x)\in A_{1}$ with $w=a \cdot b^{1-p}$. Then if we choose
\begin{equation*}
  w_{0}(x)=a(x)b(x)^{1-p_{0}},\qquad
  w_{1}(x)=a(x)b(x)^{1-p_{1}}
\end{equation*}
we see that \eqref{eq:choicew} is satisfied, and of course
$w_{j}\in A_{p_{j}}$. This concludes the proof.

% section fractional_powers (end)

%%%%%%%%%%%%%%%%%%%%   APPENDIX                              (fold)

\appendix  %% Changes numbering and naming of sections (A, B,...)

\section{Proof of Lemma \ref{lem:auscapp}}\label{sec:appendix}

The following proof follows \cite{AuscherMartell07-a} closely,
with some minor modifications and simplifcations
as explained at the beginning of Section \ref{sec:bddfunct}.
We denote by $\one{A}$ the characteristic function of 
a set $A$, and, given a ball $B$, by $mB$ the ball
with the same center and radius multiplied by a factor $m$.
Consider the sets
\begin{equation*}
  U_{\lambda}=\{MF>K \lambda,\ G\le \gamma \lambda\}\ \ \subseteq\ \ 
  E_{\lambda}=\{MF>\lambda\}.
\end{equation*}
$E_{\lambda}$ is open and we can decompose it in a sequence of
disjoint Whitney cubes $E=\bigcup_{j} Q_{j}$ with
$4Q_{j}\cap (\mathbb{R}^{n}\setminus E_{\lambda})\neq \emptyset$,
so that
\begin{equation}\label{eq:xj}
  \exists
  x_{j}\in 4 Q_{j} \quad\text{with}\quad
  MF(x_{j})\le \lambda.
\end{equation}
To each $Q_{j}$ we associate a ball $B_{j}$ with the same center as 
$Q_{j}$ and radius equal to $16$ times the side of $Q_{j}$.
Clearly we have also $U_{\lambda}=\bigcup_{j}E_{\lambda}\cap Q_{j}$.
In the following we shall discard the cubes such that
$U_{\lambda}\cap Q_{j}=\emptyset$, and select an arbitrary
$y_{j}\in U_{\lambda}\cap Q_{j}$, so that
\begin{equation}\label{eq:yj}
  y_{j}\in Q_{j},\qquad
  MF(y_{j})>K \lambda,\qquad 
  G(y_{j})\le \gamma \lambda.
\end{equation}

We remark that from the above choices it follows
\begin{equation}\label{eq:MFcutoff}
  |\{MF>K \lambda\}\cap Q_{j}|\le
  |\{M(F\one{B_{j}})>K \lambda/2\}|.
\end{equation}
Indeed, take any point 
$x\in \{MF>\lambda\}\cap Q_{j}$ and a ball $B$ containing $x$ with
$\int_{B}|F|>K \lambda|B|$. If $B \subseteq B_{j}$ we have
\begin{equation*}
  \int_{Q\cap B_{j}}|F|=\int_{B}|F|>K \lambda|B|
  \ \implies\ 
  M(F\one{B_{j}})(x)>K \lambda;
\end{equation*}
if on the other hand $B\not \subseteq B_{j}$, it is easy to chack 
that $2B$ must contain $x_{j}$ and this implies
(recalling that $MF(x_{j})\le \lambda$)
\begin{equation*}
  \int_{B \setminus B_{j}}|F|\le \int_{2B}|F|\le \lambda |2B|
\end{equation*}
so that, using $K\ge 2^{n+2}a\ge 2^{n+2}$,
\begin{equation*}
  \int_{B\cap B_{j}}|F|> K \lambda|B|-|2 B|\lambda\ge
  (K-2^{n}) \cdot|B\cap B_{j}|\cdot\lambda
  \ge
  \frac{K \lambda} 2 \cdot|B\cap B_{j}|.
\end{equation*}

In order to prove inequality \eqref{eq:lemAM}, we rewrite it as
\begin{equation*}
  w(U_{\lambda})\le 
  \|w\|_{RH_{s'}}
  C_{0}\cdot
  \left(
    \frac{\gamma}{K}+\frac{a^{q}}{K^{q}}
  \right)^{\frac1s}
  \cdot
  w(E_{\lambda})
\end{equation*}
which is implied by
\begin{equation*}
  w(U_{\lambda}\cap Q_{j})\le 
  \|w\|_{RH_{s'}} 
  C_{0}\cdot
  \left(
    \frac{\gamma}{K}+\frac{a^{q}}{K^{q}}
  \right)^{\frac1s}
  \cdot
  w(Q_{j})\quad\text{for every $j$}.
\end{equation*}
Thus, recalling \eqref{eq:RHineq}, we see that it is sufficient
to prove
\begin{equation}\label{eq:suffQj}
  |U_{\lambda}\cap Q_{j}|\le
  C_{0}\cdot
  \left(
    \frac{\gamma}{K}+\frac{a^{q}}{K^{q}}
  \right)
  |Q_{j}|\quad\text{for every $j$}.
\end{equation}

Now, by \eqref{eq:MFcutoff}, we can write
\begin{equation*}
  |U_{\lambda}\cap Q_{j}|\le|\{MF>K \lambda\}\cap Q_{j}|\le
  |\{M(F\one{B{j}})>K \lambda/2\}|
\end{equation*}
and using 
$F\one{B_{j}}\le G_{B_{j}}\one{B_{j}}+H_{B_{j}}\one{B_{j}}$ 
we obtain
\begin{equation}\label{eq:twopieces}
  |U_{\lambda}\cap Q_{j}|\le
   |\{M(G_{B_{j}}\one{B_{j}})>K \lambda/4\}|+
   |\{M(H_{B_{j}}\one{B_{j}})>K \lambda/4\}|=
   I+II.
\end{equation}
To the term $I$ we apply the weak bound \eqref{eq:Mqqbound}
for $q=1$:
\begin{equation}\label{eq:firstMF}
  |\{M(G_{B_{j}}\one{B_{j}})>K \lambda/4\}|\le 
  \frac{4c_{1}}{K \lambda}\int_{B_{j}}|G_{B_{j}}|\le 
  \frac{4c_{1}}{K \lambda}|B_{j}|G(y_{j})\le
  \frac{2^{5n+2}c_{1}}{K}|Q_{j}|\gamma
\end{equation}
where we used \eqref{eq:cond3}, \eqref{eq:yj} and
$|B_{j}|\le 2^{5n}|Q_{j}|$.

Consider then the term $II$ in \eqref{eq:twopieces}. When $q=\infty$
we can write by \eqref{eq:cond2}, \eqref{eq:xj}, \eqref{eq:yj}
and $K\ge 2^{n+1}a$
\begin{equation*}
  \|M(H_{B_{j}}\one{B_{j}})\|_{L^{\infty}}\le
  \|H_{B_{j}}\one{B_{j}}\|_{L^{\infty}}\le
  a(MF(x_{j})+MG(y_{j}))\le 2a \lambda\le \frac{K \lambda}{4}
\end{equation*}
so that $II \equiv0$. When $q<\infty$, we use the 
weak $(q,q)$ bound \eqref{eq:Mqqbound}, \eqref{eq:cond2} and 
\eqref{eq:xj} to obtain
\begin{equation*}
  II\le \frac{4^{q}c_{q}}{(K \lambda)^{q}}
  \|H_{B_{j}}\|_{L^{q}(B_{j})}^{q}\le
  \frac{4^{q}c_{q}}{(K \lambda)^{q}}
           \cdot |B_{j}|\cdot a^{q}[MF(x_{j})+G(y_{j})]^{q}\le
  \frac{2^{5(n+q)}c_{q}a^{q}}{K ^{q}}|Q_{j}|
\end{equation*}
which together with \eqref{eq:firstMF} implies \eqref{eq:suffQj}
and concludes the proof of \eqref{eq:lemAM}.

We now prove \eqref{eq:1lemAM}; we can assume
that the right hand side is finite. First we choose $K$ large enough
and $\gamma$ small enough that
\begin{equation*}
  C_{0}\cdot
  \left(
    \frac{\gamma}{K}+\frac{a^{q}}{K^{q}}
  \right)^{\frac1s}
  \cdot
  \|w\|_{RH_{s'}}\le
  \frac{1}{2K^{p}};
\end{equation*}
to obtain this, it is sufficient to set
\begin{equation}\label{eq:Kaga}
  K^{q-ps}=4^{s}(C_{0}\|w\|_{RH_{s'}}+2^{n})^{s}a^{q},\qquad
  \gamma=4^{-s}(C_{0}\|w\|_{RH_{s'}}+2^{n})^{-s}\cdot K^{1-ps}.
\end{equation}
With this choice, \eqref{eq:lemAM} implies
(after a rescaling $\lambda\to \lambda/K$)
\begin{equation}\label{eq:Mrescal}
  w\{MF>\lambda\}\le \frac1{2K^{p}} w\{MF>\lambda/K\}+ 
         w\{MG>\gamma\lambda/K\}.
\end{equation}
Now define, for $j\in \mathbb{Z}$,
\begin{equation*}
  c_{j}=\int_{K^{j}}^{K^{j+1}} \!\!\!\!\!\!\!
       p \lambda^{p}
      w\{MF>\lambda\}\frac{d \lambda}{\lambda},\qquad
  d_{j}=\int_{\gamma K^{j-1}}^{\gamma K^{j}} \!\!\!\!\!\!\!
       p \lambda^{p}
      w\{MG>\lambda\}\frac{d \lambda}{\lambda}.
\end{equation*}
Multiplying \eqref{eq:Mrescal} by $p \lambda^{p}$ and integrating
in $d \lambda/\lambda$ we obtain that $c_{j}$, $d_{j}$
are finite and satisfy
\begin{equation}\label{eq:cjdj}
  c_{j}\le \frac12 c_{j-1}+\left(\frac{K}{\gamma}\right)^{p}d_{j}. 
\end{equation}
Summing from $-N$ to $N$, $N>0$, we have, with $C'=(K/\gamma)^{p}$,
\begin{equation*}
  \sum_{-N}^{N}c_{j}\le 
  \frac12\sum_{-N-1}^{N-1}c_{j}+C'\sum_{-N}^{N}d_{j}\le
     \frac12\sum_{-N}^{N}c_{j}+\frac12c_{-N-1}+
     C'\sum_{-N}^{N}d_{j}
\end{equation*}
and hence
\begin{equation*}
  \sum_{-N}^{N}c_{j}\le c_{-N-1}+2C'\sum_{-N}^{N}d_{j}
  \ \implies\ 
  \sum_{-\infty}^{+\infty}c_{j}\le 
  \limsup_{j\to-\infty}c_{j}
  +2C'\sum_{-\infty}^{+\infty}d_{j}.
\end{equation*}
If we can show that $c_{j}$ is uniformly bounded for $j<0$,
this implies that the series in $c_{j}$ converges and
hence the limsup is actually $0$, implying
\begin{equation*}
  \sum_{-\infty}^{+\infty}c_{j}\le 
  2\left(\frac{K}{\gamma}\right)^{p}\sum_{-\infty}^{+\infty}d_{j}
\end{equation*}
which gives \eqref{eq:1lemAM} and concludes the proof.
The bound on $c_{j}$ is easy if the weight $w$ is an $L^{\infty}$ 
function: using the weak $(1,1)$ estimate for $MF$ we have
\begin{equation*}
  c_{j}\le \|w\|_{L^{\infty}}\|F\|_{L^{1}}
  \int_{K^{j-1}}^{K^{j}}p \lambda^{p-1}d \lambda
\end{equation*}
which is bounded uniformly for $j<0$ since $K>1$ and $p\ge1$.
If $w$ is not in $L^{\infty}$, we first prove the estimate for the
truncated weight $w_{R}=\inf\{w,R\}$ for all $R>0$, then
observe that the constant in the estimate depends only
on the quantity $\|w_{R}\|_{RH_{s'}}$, which is bounded
uniformly in $R\ge1$ since $w\in RH_{s'}$, and does not depend on the
$L^{\infty}$ norm of the weight.
Letting $R\to \infty$ we obtain \eqref{eq:1lemAM}.

%(end)
%%%%%%%%%%%%%%%%%%%%   REFERENCES                            (fold)

%%% For use with Bibtex:

% \bibliographystyle{plain}
% \bibliography{/Users/piero/Documents/Biblioteca/-bib/bibliodatabase.bib}

\end{document}